\documentclass[12pt]{amsart}

\usepackage{hyperref}
\usepackage{amssymb,latexsym, amsthm, enumerate, mathptmx}
\usepackage[pagewise, mathlines]{lineno}

\usepackage{tikz}
\usetikzlibrary{arrows}

\newcommand*\patchAmsMathEnvironmentForLineno[1]{%
  \expandafter\let\csname old#1\expandafter\endcsname\csname #1\endcsname
  \expandafter\let\csname oldend#1\expandafter\endcsname\csname end#1\endcsname
  \renewenvironment{#1}%
     {\linenomath\csname old#1\endcsname}%
     {\csname oldend#1\endcsname\endlinenomath}}% 
\newcommand*\patchBothAmsMathEnvironmentsForLineno[1]{%
  \patchAmsMathEnvironmentForLineno{#1}%
  \patchAmsMathEnvironmentForLineno{#1*}}%
\AtBeginDocument{%
\patchBothAmsMathEnvironmentsForLineno{equation}%
\patchBothAmsMathEnvironmentsForLineno{align}%
\patchBothAmsMathEnvironmentsForLineno{flalign}%
\patchBothAmsMathEnvironmentsForLineno{alignat}%
\patchBothAmsMathEnvironmentsForLineno{gather}%
\patchBothAmsMathEnvironmentsForLineno{multline}%
}

\newtheorem{theorem}{Theorem}[section]
\newtheorem{lemma}[theorem]{Lemma}
\newtheorem{proposition}[theorem]{Proposition}
\newtheorem{corollary}[theorem]{Corollary}

\newtheorem{conjecture}[theorem]{Conjecture}

\newtheoremstyle{definition}% name
  {6pt}%      Space above
  {6pt}%      Space below
  {}%         Body font
  {}%         Indent amount (empty = no indent, \parindent = para indent)
  {\bfseries}% Thm head font
  {.}%        Punctuation after thm head
  {.5em}%     Space after thm head: " " = normal interword space;
  {}%

\theoremstyle{definition}

\newtheoremstyle{remark}% name
  {6pt}%      Space above
  {6pt}%      Space below
  {}%         Body font
  {}%         Indent amount (empty = no indent, \parindent = para indent)
  {\bfseries}% Thm head font
  {.}%        Punctuation after thm head
  {.5em}%     Space after thm head: " " = normal interword space;
  {}%

\theoremstyle{remark}
\newtheorem{remark}[theorem]{Remark}

\newtheoremstyle{example}% name
  {6pt}%      Space above
  {6pt}%      Space below
  {}%         Body font
  {}%         Indent amount (empty = no indent, \parindent = para indent)
  {\bfseries}% Thm head font
  {.}%        Punctuation after thm head
  {.5em}%     Space after thm head: " " = normal interword space;
  {}%

\theoremstyle{example}
\newtheorem{example}[theorem]{Example}

\makeatletter
\renewcommand\@makefntext[1]{%
\setlength\parindent{1em}%
\noindent
\makebox[1.8em][r]{}{#1}}
\makeatother

\title{Broken circuit complexes of series-parallel networks}

\author{Dinh Van Le}
\address{Universit\"at Osnabr\"uck, Institut f\"ur Mathematik, 49069 Osnabr\"uck, Germany}
\email{dlevan@uos.de}

\begin{document}

\dedicatory{Dedicated to Professor Duong Quoc Viet on the occasion of his 60th birthday}

\begin{abstract}
Let $(h_0,h_1,\ldots,h_s)$ with $h_s\ne0$ be the $h$-vector of the broken circuit complex of a series-parallel network $M$. Let $G$ be a graph whose cycle matroid is $M$. We give a formula for the difference $h_{s-1}-h_1$ in terms of an ear decomposition of $G$. A number of applications of this formula are provided, including several bounds for $h_{s-1}-h_1$, a characterization of outerplanar graphs, and a solution to a conjecture on $A$-graphs posed by Fenton. We also prove that $h_{s-2}\geq h_2$ when $s\geq 4$.
\end{abstract}
\maketitle

\section{Introduction}

%\linenumbers

Let $M$ be a loopless matroid of rank $r$ on the ground set $E$. For a linear ordering $<$ of $E$, the \emph{broken circuit complex} of $M$ with respect to $<$, denoted by $BC(M,<)$, is the family of those subsets of $E$ that do not contain a \emph{broken circuit}, i.e. a circuit of $M$ with the least element deleted. Introduced by Whitney \cite{Wh} and developed further by Rota \cite{Ro}, Wilf \cite{W}, and Brylawski \cite{Br}, the broken circuit complex is an essential tool in the study of various important combinatorial and homological properties of matroids and hyperplane arrangements; see, e.g. \cite{B,BZ,BrOx,EPY,KR,Le,LR}. An interesting feature of the broken circuit complex is that its \emph{$f$-vector} $f=(f_0,\ldots,f_r)$, where $f_i$ is the number of faces of $BC(M,<)$ of cardinality $i$, encodes the coefficients of the characteristic polynomial of the matroid \cite{Ro}: $\chi(M;x)=\sum_{i=0}^r(-1)^if_ix^{r-i}.$ Given that the characteristic polynomial has a large number of diverse applications (
such as in the study of the critical problem, linear codes, hyperplane arrangements, separation of points by hyperplanes, series-parallel networks, colorings and flows in graphs, and orientations of graphs; see \cite{BrOx2} and \cite{Za} for surveys), information about the $f$-vector of the broken circuit complex could be used to solve many combinatorial problems. Therefore, the $f$-vector of the broken circuit complex is one of the most interesting numerical invariants in matroid theory.

\footnotetext{
	\begin{itemize}
		\item[ ]
		{\it Mathematics  Subject  Classification} (2010): 05B35, 05C83.
		\item[ ]
		{\it Key words and phrases}: broken circuit complex, $h$-vector, matroid, series-parallel network.
	\end{itemize}
}

In this paper we will be concerned with the \emph{$h$-vector}, an invertible linear transformation of the $f$-vector, of the broken circuit complex (see Section \ref{se2.3} for the precise definition). It should be noted that although the $h$-vector and $f$-vector encode the same information, certain properties of the broken circuit complex (such as the Gorenstein and complete intersection properties; see \cite{Le}) are better expressed through the $h$-vector. Let $(h_0(M),\ldots,h_s(M))$ be the $h$-vector of the broken circuit complex of $M$ with the zero entries at the end removed. (The index $s$ is the largest such that $h_{s}(M)\ne 0$. If $M$ is connected then $s=r-1$.) For our purposes it is  convenient to introduce a related vector. Let $\delta_i(M)=h_{s-i}(M)-h_i(M)$ for $i=0,\ldots,\lfloor s/2\rfloor$ and $\delta_i(M)=0$ for $i>\lfloor s/2\rfloor$. We call $(\delta_0(M),\delta_{1}(M),\ldots)$ the \emph{$\delta$-vector} of (the broken circuit complex of) $M$. For the sake of brevity, we will use 
throughout the paper some further notation. Let $\mathcal{S}$ be the class of loopless matroids $M$ with $\delta_0(M)=0$. Members of $\mathcal{S}$ are matroids whose connected components are series-parallel networks. For $i\geq0$, denote by $\mathcal{S}_i$ the subclass of $\mathcal{S}$ consisting of matroids $M$ with $\delta_1(M)=i$. Furthermore, we set $\mathcal{S}_{1^+}=\mathcal{S}-\mathcal{S}_0$.

This paper serves two purposes. The first one has its root in \cite{Le}, in which it is proved that $M$ admits a complete intersection broken circuit complex if and only if $M\in\mathcal{S}_0$. (Recall that a simplicial complex is a complete intersection if its minimal non-faces are pairwise disjoint.) This led us to the following observation: for $M\in\mathcal{S}$, the number $\delta_1(M)$ might have significant implications for the structure of $M$. As the main result of the paper, we make this idea precise by giving a formula for $\delta_1(M)$, when $M$ is a series-parallel network, in terms of an ear decomposition of a (graphical) series-parallel network $G$ whose cycle matroid is $M$. Let us briefly describe this formula. Let $\Pi=(\pi_1,\pi_2,\ldots,\pi_n)$ be an ear decomposition of $G$. Thus $\Pi$ is a partition of the edges of $G$, in which $\pi_1$ is a cycle and for each $i\geq 2$, $\pi_i$ is a path whose end vertices both belong to some $\pi_j$ with $j<i$. When the end vertices of $\pi_i$ are in $\pi_j$ and at least one of them is an internal vertex of $\pi_j$, the nest interval of $\pi_i$ in $\pi_j$ is the subpath of $\pi_j$ between the end vertices of $\pi_i$. For each nest interval $I$, let $\sigma^+(I)=\{I\}\cup\sigma(I)$, where $\sigma(I)$ is the set of all $\pi_i\in\Pi$ whose nest interval is $I$. Denote by $\ell(I)$ the minimal length of a path in $\sigma^+(I)$. Let $p_1(G;\Pi)$ and $p_2(G;\Pi)$ be the number of nest intervals $I$ such that $\ell(I)=1$ and $\ell(I)>1$, respectively. Then we have the following formula
\begin{equation}\label{eq0}
    \delta_1(M)=p_2(G;\Pi). 
\end{equation}
This formula, which will be proved in Theorem \ref{th45}, has plenty of applications. We first derive in Section 4 several bounds for $\delta_1(M)$: an upper bound in terms of $h_1(M)$ (Proposition \ref{pr411}) and lower and upper bounds in terms of the number of vertices of degree at least 3 of $G$ (Proposition \ref{pr414}). Further applications of formula \eqref{eq0} are given in Section 5. We show that members of $\mathcal{S}_1$ are essentially cycle matroids of subdivisions of complete bipartite graphs $K_{2,m}$ for $m\geq 3$ (Proposition \ref{pr51}). We also show that any member of $\mathcal{S}_{1^+}$ possesses a parallel minor in $\mathcal{S}_1$ (Proposition \ref{pr54}). These results together with a characterization of $\mathcal{S}$ due to Brylawski \cite{Br2} give excluded minors for $\mathcal{S}_0$: its members have no minor isomorphic to $U_{2,4}$, $M(K_4)$ and no parallel minor isomorphic to $M(K_{2,m})$ for $m\geq3$ (Theorem \ref{th57}). On specializing to graphs, it is proved that for a graph 
$G$, $M(G)\in \mathcal{S}_0$ if and only if $G$ contains no subgraph that is a subdivision of $K_4$ and the simplification $\overline{G}$ of $G$ contains no vertex-induced subgraph that is a subdivision of $K_{2,3}$ (Theorem \ref{th58}). From this latter result we derive two graph-theoretic consequences: a characterization of outerplanar graphs (Corollary \ref{co59}) and a solution to a conjecture on $A$-graphs posed by Fenton \cite{Fe} (Corollary \ref{co511}).

Formula \eqref{eq0} shows that the number $p_2(G;\Pi)$ is independent of the decomposition $\Pi$. As a counterpart of this formula, we prove in Theorem \ref{th42} that the number $p_1(G;\Pi)$ is also independent of $\Pi$, and furthermore, $p_1(G;\Pi)$ brings information about parallel irreducible decompositions of $G$. This result together with formula \eqref{eq0} yields several characterizations as well as sufficient conditions for the parallel irreducibility of a series-parallel network (Corollaries \ref{lm43}, \ref{co410}, \ref{co413}, Propositions \ref{pr411}, \ref{pr414}).

The second purpose of this paper is to study the following conjecture (see \cite{Sw}):

\begin{conjecture}\label{con11}
The $\delta$-vector of the broken circuit complex of an arbitrary matroid is nonnegative.
\end{conjecture}
This conjecture is related to a long-standing conjecture of Stanley on $h$-vectors of independence complexes and a weaker version thereof due to Hibi. The \emph{independence complex} (or \emph{matroid complex}) of a matroid is the collection of all independent sets of the matroid. This complex contains the broken circuit complex as a subcomplex. In \cite{St1}, Stanley conjectured that the $h$-vector of an independence complex is a \emph{pure $O$-sequence}, i.e. the degree sequence of an order ideal of monomials all of whose maximal elements have the same degree (see also \cite{St2} for more details). A pure $O$-sequence $(h_0,h_1,\ldots,h_s)$ with $h_s\ne0$ has the following properties
%\begin{linenomath*}
\begin{align}
&h_0\leq h_1\leq \cdots\leq h_{\lfloor s/2\rfloor}, \label{eq1:2}\\
&h_i\leq h_{s-i} \text{ for } i=0,\ldots,\lfloor s/2\rfloor. \label{eq1:3}
\end{align}
%\end{linenomath*}
This result was proved by Hibi \cite{Hi1}, and it led him to propose the following weaker version of Stanley's conjecture \cite{Hi2}: the $h$-vector of an independence complex satisfies inequalities \eqref{eq1:2} and \eqref{eq1:3}. In order to resolve Hibi's conjecture, Chari \cite{C} introduced the notion of \emph{convex ear decomposition} of simplicial complexes, which can be viewed as a higher-dimensional analogue of the notion of ear decomposition of graphs. He showed that the $h$-vectors of simplicial complexes that admit a convex ear decomposition satisfy inequalities \eqref{eq1:2} and \eqref{eq1:3}, and that the independence complex of every coloopless matroid admits such a decomposition, thereby settling Hibi's conjecture.

Note that the set of $h$-vectors of independence complexes is a (strict) subset of the set of $h$-vectors of broken circuit complexes, since the cone on any independence complex is the broken circuit complex of another matroid \cite{Br}. In this context, Conjecture \ref{con11} is an extension of inequality \eqref{eq1:3} in Hibi's conjecture. (We remark that inequality \eqref{eq1:2} for the $h$-vectors of broken circuit complexes would follow from Conjecture \ref{con11} and unimodality of those $h$-vectors.  A recent important result of Huh \cite{Hu} confirms log-concavity (hence unimodality) for the $h$-vectors of broken circuit complexes of matroids representable over a field of characteristic zero.) However, it is worth mentioning that, as Chari noted in the last part of his paper, the broken circuit complex does not in general admit a convex ear decomposition. Therefore, Chari's method does not establish Conjecture \ref{con11}.

In the final section of this paper we present some results motivated by Conjecture \ref{con11}. For general matroids $M$ it is only known that the first two entries $\delta_0(M),\delta_1(M)$ of the $\delta$-vector of $M$ are nonnegative \cite[Section 5]{Sw}. In the case $M\in\mathcal{S}$ we will show that $\delta_2(M)\geq0$ (Theorem \ref{th61}). We also verify Conjecture \ref{con11} for $M\in\mathcal{S}_1$ (Proposition \ref{pr62}).

In order to make the paper self-contained, we include in Section 2 the relevant notions and facts concerning matroids, series-parallel networks, and broken circuit complexes. Section 3, which serves as preparation for Section 4, examines the effect on the number $\delta_1(M)$ of the contraction operation.

\section{Background}

\subsection{Matroids}

We mostly follow Oxley's book \cite{O} for matroid terminology. A {\it matroid} $M=(E,\mathcal{I})$ consists of a non-empty finite \emph{ground set} $E$ and a collection $\mathcal{I}$ of subsets of $E$, called \emph{independent sets}, such that:
\begin{enumerate}
\item
$\emptyset\in\mathcal{I};$
\item
subsets of independent sets are independent;
\item
for every subset $X$ of $E$, all maximal independent subsets of $X$ have the same cardinality $r(X)$, called the \emph{rank} of $X$.
\end{enumerate}

We call a maximal independent set of $M$ a \emph{basis}. Clearly, the matroid $M$ is specified by its bases. The rank $r(E)$ of $E$, which is the common cardinality of the bases, is also called the \emph{rank} of $M$ and is denoted by $r(M)$. A subset of $E$ is \emph{dependent} if it is not in $\mathcal{I}$. Minimal dependent sets are called \emph{circuits}. An element $e\in E$ is a \emph{loop} if $\{e\}$ is a circuit of $M$. A circuit of cardinality $m$ is called an $m$-circuit. Note that the family $\mathcal{C}(M)$ of circuits also determines the matroid $M$: $\mathcal{I}$ consists of subsets of $E$ that do not contain any member of $\mathcal{C}(M)$.

A typical example of a matroid is the \emph{vector matroid} of a matrix $A$ over some field: the ground set $E$ is the set of column vectors of $A$ and the independent sets are the linearly independent subsets of $E$. Another common example is the cycle matroid of a graph: Let $G$ be a graph whose edge set is $E.$ Then the collection of edge sets of cycles of $G$ forms the family of circuits of a matroid $M(G)$ on $E$. We call $M(G)$ the \emph{cycle matroid} of $G$. The bases of $M(G)$ are the edge sets of spanning forests of $G$. Thus, in particular, $r(M(G))=|V|-\omega$, where $V$ and $\omega$ are respectively the vertex set and the number of connected components of $G$. A matroid is called \emph{graphic} if it is isomorphic to the cycle matroid of a graph. (Two matroids $M_1,M_2$ on ground sets $E_1,E_2$ are \emph{isomorphic} if there exists a bijection $\varphi:E_1\to E_2$ such that $X\subseteq E_1$ is independent in $M_1$ if and only if $\varphi(X)$ is independent in $M_2$.) In this paper, we will also deal with uniform matroids which are defined as follows: For nonnegative integers $m\leq n$, the \emph{uniform matroid} $U_{m,n}$ is the matroid on an $n$-element ground set $E$ whose independent sets are the subsets of $E$ of cardinality at most $m$. So the circuits of $U_{m,n}$ are the $(m+1)$-element subsets of $E$. In particular, when $n=m+1$, the matroid $U_{m,m+1}$ has a unique circuit $C_{m+1}=E$. Identifying $U_{m,m+1}$ with $C_{m+1}$, by the term ``circuit'' we will sometimes mean ``matroid with a unique circuit''.

Let $M$ be a matroid on the ground set $E$. The \emph{dual} $M^*$ of $M$ is the matroid on the ground set $E$ whose bases are the complements of the bases of $M$. For example, $U_{m,n}^*=U_{n-m,n}$. It is well-known that if $G$ is a vl.,he
ar graph, then $M(G)^*\cong M(G^*)$, where $G^*$ is a geometric dual of $G$. The loops of $M^*$ are called \emph{coloops} of $M$. Clearly, $e$ is a coloop of $M$ if and only if $e$ is contained in every basis of $M$ if and only if $e$ is not contained in any circuit of $M$.

Let $X$ be a subset of $E$. The \emph{deletion} of $X$ from $M$, denoted $M-X$, is the matroid on the ground set $E-X$ whose circuits are those members of $\mathcal{C}(M)$ which are contained in $E-X$. The \emph{contraction} of $X$ from $M$ is given by $M/X=(M^*-X)^*$. One may check that the circuits of $M/X$ are the minimal non-empty members of $\{C-X:C\in\mathcal{C}(M)\}$. Note that deletion and contraction for matroids generalize the corresponding operations for graphs, i.e. $M(G)-X=M(G-X)$ and $M(G)/X=M(G/X)$ for any graph $G$ and any set $X$ of edges of $G$. Note also that the operations of deletion and contraction commute: for disjoint subsets $X$ and $Y$ of $E$, one has $(M-X)/Y=M/Y-X$. A \emph{minor} of $M$ is a matroid which can be obtained from $M$ by a sequence of deletions and contractions. So every minor of $M$ has the form $(M-X)/Y$, where $X$, $Y$ are disjoint subsets of $E$.

Two elements $e,f\in E$ are said to be \emph{parallel} if they form a circuit of $M$. A \emph{parallel class} of $M$ is a maximal subset of $E$ in which any two distinct elements are parallel and no element is a loop. Obviously, if $X$ is a parallel class of $M$, then for any $e\in X$ every element of $X-e$ is a loop in $M/e$. Conversely, if $X$ contains no loops and for some $e\in X$ every element of $X-e$ is a loop in $M/e$, then $X$ is contained in a parallel class of $M$. A parallel class of $M^*$ is called a \emph{series class} of $M$. If $Y$ is a series class of $M$, then for any $e\in Y$ every element of $Y-e$ is a coloop in $M-e$. (Thus, for any series class $Y$ and any circuit $C$ of $M$, either $Y\subset C$ or $Y\cap C=\emptyset$.) Conversely, if $Y$ is contained in a circuit of $M$ and for some $e\in Y$ every element of $Y-e$ is a coloop in $M-e$, then $Y$ is a subset of a series class of $M$. A parallel or series class is \emph{non-trivial} if it contains at least two elements. A matroid is
called \emph{simple} if it has no loops and no non-trivial parallel classes. Given  an arbitrary matroid $M$
we may associate to it a simple matroid by first deleting all the loops from $M$ and then, for every parallel class $X$, deleting all but one distinguished element of $X$. The matroid obtained, denoted by $\overline{M}$, is uniquely determined up to the choice of the distinguished elements and is called the \emph{simplification} of $M$. Evidently, one may also construct the \emph{simplification} $\overline{G}$ of a given graph $G$ in the same manner as above, and moreover, one has $\overline{M(G)}=M(\overline{G}).$

Let $M_1$ and $M_2$ be matroids on disjoint sets $E_1$ and $E_2$. Their \emph{direct sum} $M_1\oplus M_2$ is the matroid on the ground set $E_1\cup E_2$ whose circuits are the circuits of $M_1$ and the circuits of $M_2$. A matroid is called \emph{connected} if it is not the direct sum of two smaller matroids; otherwise it is \emph{disconnected}. Every matroid $M$ can be decomposed uniquely (up to order) as a direct sum $M=M_1\oplus\cdots\oplus M_k$ of connected matroids; we call $M_1,\ldots,M_k$ the \emph{connected components} of $M$. It should be noted that if $G$ is a loopless connected graph with at least 3 vertices, then the cycle matroid $M(G)$ is connected if and only if $G$ is \emph{2-connected}, i.e. $G$ remains connected after deleting any vertex (see \cite[Proposition 4.1.8]{O}).

\subsection{Series and parallel connection}

The operations of series and parallel connections of graphs have their origin in electrical-network theory. These operations were generalized to matroids by Brylawski \cite{Br2}. Here we briefly summarize some of their properties. The reader is referred to \cite{Br4} or \cite[5.4, 7.1]{O} for further details.

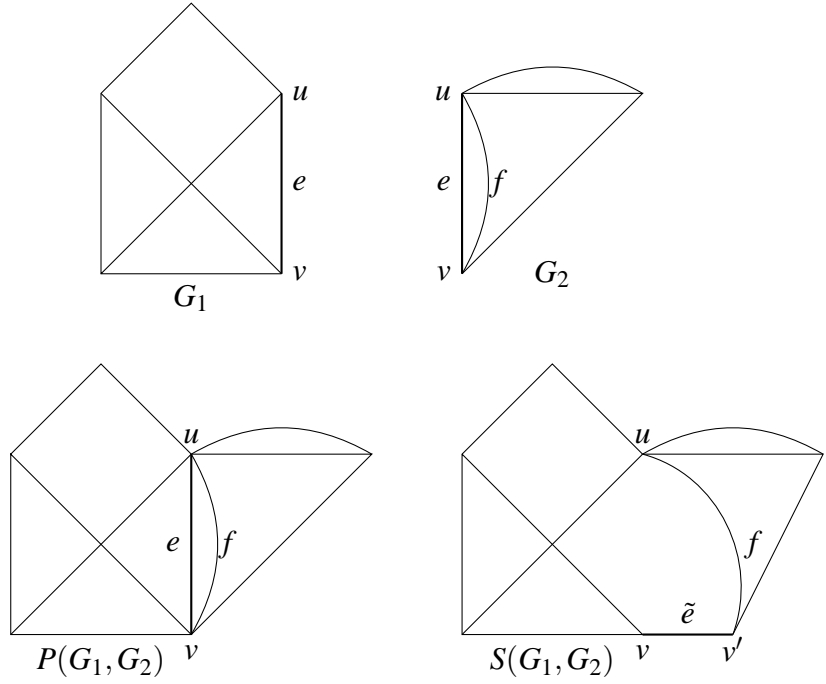
\begin{figure}[htb]
 \begin{tikzpicture}[scale=1.2]
    \draw [thick](2,0) -- (2,2);
    \draw (0,0) -- (2,2) -- (1,3) -- (0,2)--(0,0)-- (2,0);
    \draw (2,0) -- (0,2);
    \draw (2,0)  node [right] {$v$};
    \draw (2,2)  node [right] {$u$};
    \draw (2,1)  node [right] {$e$};
    \node[below] at (1,0)  {$G_1$};

    \draw (4,0) -- (6,2) -- (4,2);
    \draw [thick] (4,2)--(4,0);
    \draw (6,2) to [out=150,in=30] (4,2) to [out=300,in=60] (4,0);
    \draw (4,0)  node [left] {$v$};
    \draw (4,2)  node [left] {$u$};
    \draw (4,1)  node [left] {$e$};
    \draw (4.4,1)  node {$f$};
    \node at (5,0)  {$G_2$};

    \draw (-1,-2) -- (0,-1) -- (1,-2) -- (3,-2)--(1,-4)-- (-1,-4)--(-1,-2)--(1,-4);
    \draw (-1,-4) -- (1,-2);
    \draw [thick] (1,-2)--(1,-4);
    \draw (3,-2) to [out=150,in=30] (1,-2) to [out=300,in=60] (1,-4);
    \draw (1,-4)  node [below] {$v$};
    \draw (1,-2)  node [above] {$u$};
    \draw (1,-3)  node [left] {$e$};
     \draw (1.4,-3)  node {$f$};
    \node[below] at (0,-4)  {$P(G_1,G_2)$};

     \draw (7,-4) -- (8,-2) -- (6,-2) -- (5,-1)--(4,-2)-- (4,-4)--(6,-4)--(4,-2);
    \draw (4,-4) -- (6,-2);
    \draw [thick] (6,-4)--(7,-4);
    \draw (8,-2) to [out=150,in=30] (6,-2) to [out=345,in=70] (7,-4);
    \draw (6,-4)  node [below] {$v$};
    \draw (7,-3.9)  node [below] {$v'$};
    \draw (6,-2)  node [above] {$u$};
    \draw (6.5,-4)  node [above] {$\tilde{e}$};
    \draw (7,-3)  node [right] {$f$};
    \node[below] at (5,-4)  {$S(G_1,G_2)$};
   \end{tikzpicture}
\caption{Series and parallel connections of graphs.}\label{fig1}
\end{figure}

Let us first recall the definitions of series and parallel connections of two graphs. For $i=1,2$, let $G_i$ be a graph with vertex set $V_i$ and edge set $E_i$. Assume that $G_1$ and $G_2$ have only a common edge $e$ and two common vertices $u,v$ which are the end vertices of $e$. Then the \emph{parallel connection} of $G_1$ and $G_2$ with respect to the \emph{baseedge} $e$ is merely the union of $G_1$ and $G_2$, i.e. the graph with vertex set $V_1\cup V_2$ and edge set $E_1\cup E_2$. We denote this graph by $P(G_1,G_2)$. To define the series connection of $G_1$ and $G_2$, we first form a copy $G_2'=(V_2',E_2')$ of $G_2$ by just renaming the vertex $v$ to $v'$ and the edge $e$ to $e'$. Then we remove the edge $e$ from $G_1$ and the edge $e'$ from $G_2'$. Finally, we add a new edge $\tilde{e}$ joining $v$ and $v'$. The \emph{series connection} of $G_1$ and $G_2$ with respect to $e$, denoted $S(G_1,G_2)$, is the graph with vertex set $V_1\cup V_2'$ and edge set $(E_1-e)\cup (E_2'-e')\cup \tilde{e}$ (see 
Figure \ref{fig1}).

Next we extend the above constructions to matroids. Let $M_1$ and $M_2$ be matroids on ground sets $E_1$ and $E_2$ with $E_1\cap E_2=\{e\}$. Assume that $e$ is neither a loop nor a coloop of $M_1$ or $M_2$. As before, we use the notation $\mathcal{C}(M)$ to denote the family of circuits of a matroid $M$. The \emph{series connection} $S(M_1,M_2)$ and the \emph{parallel connection} $P(M_1,M_2)$ of $M_1,M_2$ with respect to the \emph{basepoint} $e$ are the matroids on the ground set $E_1\cup E_2$ whose families of circuits are respectively:
\[\begin{aligned}
\mathcal{C}(S(M_1,M_2))&=\mathcal{C}(M_1-e)\cup\mathcal{C}(M_2-e)\cup\{C_1\cup C_2: e\in C_i\in\mathcal{C}(M_i)\ \text{ for }\ i=1,2\},\\
  \mathcal{C}(P(M_1,M_2))&=\mathcal{C}(M_1)\cup\mathcal{C}(M_2)\cup\{C_1\cup C_2-e: e\in C_i\in\mathcal{C}(M_i)\ \text{ for }\ i=1,2\}.
  \end{aligned}\]

Moreover, we set

\[
   S(M_1,M_2)=S(M_2,M_1)=
\begin{cases}
 (M_1/e)\oplus M_2\ \text{if } e \ \text{is a loop of } M_1,\\
M_1\oplus (M_2-e)\ \text{if } e \ \text{is a coloop of } M_1;
\end{cases}\]

\[P(M_1,M_2)=P(M_2,M_1)=
\begin{cases}
 M_1\oplus (M_2/e)\ \text{if } e \ \text{is a loop of } M_1,\\
(M_1-e)\oplus M_2\ \text{if } e \ \text{is a coloop of } M_1.
\end{cases}
\]

 The above constructions of series and parallel connection of matroids generalize the corresponding ones for graphs, in the sense that for any two graphs $G_1,G_2$ one has
 \[P(M(G_1),M(G_2))\cong M(P(G_1,G_2)),\quad S(M(G_1),M(G_2))\cong M(S(G_1,G_2))\]
  whenever $P(G_1,G_2)$ and $S(G_1,G_2)$ make sense.

It is possible to define series and parallel connections of more than two matroids, just by iterating the above constructions. Let $M_1,\ldots,M_n$ be matroids on ground sets $E_1,\ldots,E_n$ such that $E_{i+1}\cap(\bigcup_{j=1}^i E_j)=\{e_i\}$ for $i=1,\ldots,n-1$, in which $e_1,\ldots,e_{n-1}$ need not be distinct. Then we can form $P(M_1,M_2)$, $P(P(M_1,M_2),M_3)$ and so on. The last matroid obtained in this way, denoted by $P(M_1,\ldots,M_n)$, is called the \emph{(iterated) parallel connection} of $M_1,\ldots,M_n$ with respect to the basepoints $e_1,\ldots,e_{n-1}$. The series connection of $M_1,\ldots,M_n$ is defined similarly. Of course, iterated series and parallel connections of graphs can also be constructed in the same manner.

Special cases of series and parallel connections are series and parallel extensions: for two matroids $M$ and $N$, we say that $M$ is a \emph{series extension} (respectively, \emph{parallel extension}) of $N$ and $N$ a \emph{series contraction} (respectively, \emph{parallel deletion}) of $M$ if $M=S(N,C_2)$ (respectively, $M=P(N,C_2)$), where $C_2$ is a 2-circuit. For example, every loopless matroid is an iterated parallel extension of its simplification. Series extension and parallel extension for graphs mean subdividing an edge and duplicating an edge, respectively.

We now mention some other notions related to series and parallel connections which will be used later. Let $M$ and $N$ be matroids. We call $N$ a \emph{series minor} (respectively, \emph{parallel minor}; \emph{series-parallel minor}) of $M$ if $N$ can be obtained from $M$ by a sequence of deletions and series contractions (respectively, contractions and parallel deletions; series contractions and parallel deletions). Evidently, $N$ is a series minor of $M$ if and only if $N^*$ is a parallel minor of $M^*$. Suppose $M$ is connected with ground set $E$. Then $M$ is called \emph{parallel irreducible at $e\in E$} if either $M$ is trivial (i.e. $|E|=1$) or $M$ is not a parallel connection of two non-trivial matroids with respect to the basepoint $e$. We say that $M$ is \emph{parallel irreducible} if it is parallel irreducible at every element of $E$.

\begin{lemma}\label{lm22}
 Let $M$ and $M'$ be matroids on ground sets $E$ and $E'$ with $E\cap E'=\{e\}$. Then the following statements hold.
\begin{enumerate}
\item
$P(M,M')$ is loopless (respectively, simple) if and only if both $M$ and $M'$ are loopless (respectively, simple).

\item If $|E|,|E'|\geq 2$, then
$S(M,M')$ (respectively, $P(M,M')$) is connected if and only if both $M$ and $M'$ are connected.

\item
$P(M,M')/e=M/e\oplus M'/e$. If $f\in E-e$, then
\[\begin{aligned}
  S(M,M')-f&=S(M-f,M'), \quad S(M,M')/f=S(M/f,M'),\\
   P(M,M')-f&=P(M-f,M'), \quad P(M,M')/f=P(M/f,M').
  \end{aligned}\]

\item
If $e$ is neither a loop nor a coloop of $M$ or $M'$, then $M=P(M,M')-(E'-e)$. Thus $M$ and $M'$ are submatroids of $P(M,M')$.

\item
Assume that $M$ is the cycle matroid of a graph $G$. If $M(G)\cong P(M_1,M_2)$, where the basepoint of the parallel connection is neither a loop nor a coloop of $M_1$ or $M_2$, then there exist subgraphs $G_1,G_2$ of $G$ such that $M_i\cong M(G_i)$ for $i=1,2$ and $G=P(G_1,G_2)$.
\end{enumerate}
From now on suppose that $M$ is a non-trivial connected matroid.
\begin{enumerate}[\rm(vi)]
 \item
$M$ is parallel irreducible at $f\in E$ if and only if $M/f$ is connected. Hence $M$ is parallel irreducible if and only if $M/f$ is connected for every $f\in E$.

\item
For every $f\in E$, the matroid $M$ can be decomposed as a parallel connection $M=P(M_1,\ldots,M_k)$ with respect to the basepoint $f$, where $M_i$ is non-trivial and parallel irreducible at $f$ for $i=1,\ldots,k$. The decomposition is unique up to a permutation of the components.

\item
$M$ has an iterated parallel decomposition $M=P(N_1,\ldots,N_s)$, where the $N_i$ are non-trivial and parallel irreducible. Moreover, if $M$ is simple, then the decomposition is unique up to a permutation of the components.
\end{enumerate}
\end{lemma}

\begin{proof}
(i) follows trivially from the definition of parallel connection. For the proof of (ii)--(iv), (vi), (vii) see \cite[Propositions 4.6, 4.7, 5.5-5.9]{Br2}; see also \cite[7.1]{O}. From (iv) one easily gets (v). Let us prove (viii). The existence of a parallel irreducible decomposition of $M$ follows by repeatedly applying (vii). We now prove the uniqueness of this decomposition when $M$ is simple. Let
$F(M)=\{f\in E: M/f\ \text{is disconnected}\}.$
The argument is by induction on $|F(M)|$.

We first show that $F(M)$ is the set of basepoints of any parallel irreducible decomposition of $M$. Indeed, consider such a decomposition $M=P(N_1,\ldots,N_s)$. Let $E_i$ be the ground set of $N_i$ and assume that $E_{i+1}\cap(\bigcup_{j=1}^i E_j)=\{e_i\}$ for $i=1,\ldots,s-1$. From (iii) (or (vi)) we immediately get $\{e_1,\ldots,e_{s-1}\}\subseteq F(M)$. If there exists $f\in F(M)-\{e_1,\ldots,e_{s-1}\}$, for which we may assume $f\in E_1$, then from (iii) we obtain $M/f=P(N_1/f,N_2,\ldots,N_s)$. Note that $N_1/f$ is connected by (vi). Note also that $|E_1|\geq 3$ since $N_1$ is simple (by (i)) and non-trivial.
 So $|E_1-f|\geq 2$ and hence it follows from (ii) that $M/f$ is connected, a contradiction. Thus we must have $F(M)=\{e_1,\ldots,e_{s-1}\}$.

Now we can ``group'' the components $N_1,\ldots,N_s$ to get a parallel decomposition of $M$ into matroids which are parallel irreducible at $e_1$: $M=P(M_1,\ldots,M_k)$. (For example, for the decomposition $M=P(N_1,\ldots,N_4)$ with $E_2\cap E_1=\{e_1\}$, $E_3\cap (E_1\cup E_2)=E_3\cap E_2=\{e_2\}$, $E_4\cap (E_1\cup E_2\cup E_3)=E_4\cap E_1=\{e_3\}$ and $e_1,e_2,e_3$ pairwise distinct, we may write $M=P(M_1,M_2)$, where $M_1=P(N_1,N_4)$ and $M_2=P(N_2,N_3)$ are parallel irreducible at $e_1$.) By (vii), the matroids $M_1,\ldots,M_k$ are uniquely determined. Moreover, these matroids are simple by (i). Observe that $F(M_i)\subseteq F(M)-e_1$ for $i=1,\ldots,k$. Since each $M_i$ is a parallel connection of some of the matroids $N_1,\ldots,N_s$ (and this is of course a parallel irreducible decomposition of $M_i$), the uniqueness of $N_1,\ldots,N_s$ follows by induction.
\end{proof}

\begin{remark}\label{rm22}
(i) By Lemma \ref{lm22}(v), results and notions on matroids involving only parallel connection can be easily specialized to graphs. In particular, we will also use the terms ``parallel irreducible'', ``parallel irreducible decomposition'' for graphs without explicit explanation.

(ii) When a connected matroid $M$ is not simple, the parallel irreducible decomposition of $M$ as in Lemma \ref{lm22}(viii) is no longer unique. For instance, if $M$ contains a 2-circuit $C_2=\{e_1,e_2\}$, then $M=P(M_1,C_2)=P(M_2,C_2)$, where $M_i=M-e_i$ for $i=1,2$. However, it is not hard to show that if  $M=P(N_1,\ldots,N_s)$ and $M=P(N_1',\ldots,N_t')$ are two parallel irreducible decompositions of $M$, then $s=t$ and (after reindexing) $N_i\cong N_i'$ for $i=1,\ldots,s$.
\end{remark}

A graph $G$ is called a \emph{(graphical) series-parallel network} if it is a \emph{block} (i.e. a connected graph whose cycle matroid is connected) and can be obtained from the complete graph $K_2$ by subdividing and duplicating edges. Extending this notion to matroids, we call a connected matroid $M$ a \emph{series-parallel network} if it can be constructed from a coloop by a sequence of series and parallel extensions.
Clearly, a matroid is a series-parallel network if and only if it is the cycle matroid of a graphical series-parallel network. We list several characterizations of series-parallel networks in the next lemma.

\begin{lemma}\label{lm24}
 Let $M$ be a loopless connected matroid and $G$ a loopless graph with at least one edge. Then
\begin{enumerate}
 \item
$G$ is a series-parallel network if and only if it is a block having no subgraph that is a subdivision of $K_4$.

\item
The following conditions are equivalent:
\begin{enumerate}
 \item
$M$ is a series-parallel network;

\item
every connected minor of $M$ is a series-parallel network;

\item
for any connected minor $N$ of $M$ on the ground set $E(N)$ with $|E(N)|>2$, and any element $e\in E(N)$, either $N-e$ or $N/e$ is disconnected;

\item
$M$ has no minor isomorphic to $U_{2,4}$ or $M(K_4)$;

\item
$\beta(M)=1$.
\end{enumerate}
\end{enumerate}
\end{lemma}

In the above lemma, $\beta(M):=(-1)^{r(E)}\sum_{X\subseteq E}(-1)^{|X|}r(X)$ is the \emph{beta invariant} of a matroid $M$ on the ground set $E$. This invariant was introduced by Crapo \cite{Cr} and discussed further in, e.g. \cite{Br2,O3,Za}. For the proof of the lemma we refer to \cite[Theorem 5.4.10]{O} and \cite[Proposition 7.5, Theorem 7.6]{Br2}.

\subsection{Broken circuit complexes}\label{se2.3}

Let $M$ be a matroid with a given linear ordering $<$ of its ground set $E$. When no confusion may arise, we will briefly denote the broken circuit complex of $M$ with respect to $<$ by $BC(M)$. Note that if $M$ contains a loop, then $\emptyset$ is a broken circuit, and so $BC(M)=\emptyset$. It is, therefore, enough to consider broken circuit complexes of loopless matroids. Moreover, if necessary, one may even restrict attention to simple matroids because the broken circuit complex of a loopless matroid is isomorphic to that of its simplification; see \cite[Proposition 7.4.1]{B}.

Now suppose $M$ is loopless and $r(M)=r$. Then $BC(M)$ is an $(r-1)$-dimensional shellable simplicial complex; see \cite{Pr} or \cite[7.4]{B}. Let $(f_0,\ldots,f_r)$ be the $f$-vector of $BC(M)$. Then the polynomial $f(M;x)=\sum_{i=0}^rf_ix^{r-i}$ is called the \emph{$f$-polynomial} of $BC(M)$. The $h$-vector $(h_0,\ldots,h_r)$ and \emph{$h$-polynomial} $h(M;x)=\sum_{i=0}^rh_ix^{r-i}$ of $BC(M)$ are determined by the polynomial identity $h(M;x+1)=f(M;x)$. In other words, the $f$-vector and $h$-vector are correlated as follows
\[f_i=\sum_{j=0}^i\binom{r-j}{i-j}h_j \quad\text{and}\quad
 h_i=\sum_{j=0}^i(-1)^{i-j}\binom{r-j}{i-j}f_j,\quad i=0,\ldots,r.\]
%Thus, the $f$-vector and $h$-vector encode the same information. 

Observe that different orderings of the ground set $E$ of $M$ may lead to non-isomorphic broken circuit complexes; see, e.g. \cite[Example 7.4.4]{B}. However, from the formula $h(M;x)=t(M;x,0)$, where $t(M;x,y)$ is the \emph{Tutte polynomial} of $M$ defined by
\[t(M;x,y)=\sum_{X\subseteq E}(x-1)^{r(E)-r(X)}(y-1)^{|X|-r(X)}\]
(see \cite[p. 240]{B}), it follows that the $f$-vector and $h$-vector of $BC(M,<)$ are independent of the ordering $<$.

We keep the notation from the introduction. In the next lemma some properties of the $h$-vector of the broken circuit complex are summarized. Recall that a matroid is \emph{representable} if it is isomorphic to the vector matroid of a matrix over some field.

\begin{lemma}\label{lm25}
 Let $M$ be a loopless matroid on the ground set $E$ with $r(M)=r$. Let $(h_0,h_1,\ldots,h_r)$ and $h(M;x)=\sum_{i=0}^rh_ix^{r-i}$ be the $h$-vector and the $h$-polynomial of $BC(M)$, respectively. Then the following statements hold.
\begin{enumerate}
 \item
$h_i\geq 0$ for $i=0,\ldots,r$. Moreover, $h_0=1,\ h_{r-1}=\beta(M)$, and $h_r=0$. If $M$ is simple, then $h_1=|E|-r$.

 \item
Assume that $M$ is either the direct sum or the parallel connection of two matroids $M_1$ and $M_2$. Then
\[h(M;x)=
\begin{cases}
 h(M_1;x)h(M_2;x)& \text{ if } M=M_1\oplus M_2,\\
 x^{-1}h(M_1;x)h(M_2;x)& \text{ if } M=P(M_1,M_2).
\end{cases}
\]

 \item
$M$ has $k$ connected components if and only if $k$ is the smallest number such that $h_{r-k}>0$. In particular, $\beta(M)>0$ if and only if $M$ is connected.

 \item
If $e\in E$, then
\[h(M;x)=
\begin{cases}
 xh(M-e;x) \quad \text{ if } e \text{ is a coloop of } M,\\
h(M-e;x)+h(M/e;x) \quad \text{ otherwise}.
\end{cases}
\]
Thus, if $M$ is connected and $|E|\geq 2$, then either $M-e$ or $M/e$ is connected.

\item
If $M=C_{r+1}$, an $(r+1)$-circuit, then $h(M;x)=x^r+x^{r-1}+\cdots+x.$

 \item
Assume that $M$ is representable. Let $s$ be the largest index such that $h_s\ne 0$. Then $\sum_{j=0}^ih_j\leq\sum_{j=0}^ih_{s-j}$ for all $i=0,\ldots,s$.

\item
$M\in \mathcal{S}$ if and only if $M$ is a direct sum of series-parallel networks. Suppose $M\in \mathcal{S}$. Then $\delta_1(M)\geq 0$. Moreover, if $M$ is either the direct sum or the parallel connection of two matroids $M_1$ and $M_2$, then $\delta_1(M)=\delta_1(M_1)+\delta_1(M_2).$
\end{enumerate}
\end{lemma}

\begin{proof}
 The properties (i)--(iv) follow from the formula $h(M;x)=t(M;x,0)$ and the corresponding properties of the Tutte polynomial which are presented in \cite[6.2]{BrOx2} and \cite[p. 182]{Br5}. From (iv) one easily gets (v). For the proof of (vi), see \cite[Proposition 2.3(v)]{Le}. It remains to prove (vii). Let $s$ be the largest index such that $h_s\ne 0$. Then $M\in \mathcal{S}$ if and only if $h_s=1$. Note that from (ii) we obtain $h(M;x)=\prod_{i=1}^k h(N_i;x)$, where $N_1,\ldots,N_k$ are the connected components of $M$. This yields $h_s=\prod_{i=1}^k \beta(N_i)$. Thus, $h_s=1$ if and only if $ \beta(N_i)=1$ for $i=1,\ldots,k$. By Lemma \ref{lm24}, the latter condition means that all the $N_i$ are series-parallel networks. Now assume that $M\in \mathcal{S}$. Then $M$ is a graphic matroid because, as we have just shown, each $N_i$ is graphic. In particular, $M$ is representable; see \cite[Proposition 5.1.2]{O}. Hence, if $s\geq 2$, we obtain from (vi) that
 \[\delta_1(M)=h_{s-1}-h_1=(h_s+h_{s-1})-(h_0+h_1)\geq 0.\]
As $\delta_1(M)=0$ when $s\leq 1$ by definition, the second assertion follows.
To prove the last one, let $(h^{(i)}_0,\ldots,h^{(i)}_{s_i})$ with $h^{(i)}_{s_i}\ne 0$ be the $h$-vector of $BC(M_i)$ for $i=1,2$. If $M=M_1\oplus M_2$ or $M=P(M_1,M_2)$, then from the relation between $h$-polynomials given in (ii) we get
\[h_1=h^{(1)}_1+h^{(2)}_1,\quad h_{s-1}=h^{(1)}_{s_1-1}h^{(2)}_{s_2}+h^{(1)}_{s_1}h^{(2)}_{s_2-1}, \quad
h_s=h^{(1)}_{s_1}h^{(2)}_{s_2}.\]
Since $h_s=1$, $h^{(1)}_{s_1}=h^{(2)}_{s_2}=1$. Therefore,
\[\delta_1(M)=h_{s-1}-h_1=(h^{(1)}_{s_1-1}-h^{(1)}_1)+(h^{(2)}_{s_2-1}-h^{(2)}_1)=\delta_1(M_1)+\delta_1(M_2).
\vspace{-1.3\baselineskip}\]
\end{proof}

  Let us recall the following characterization of the class $\mathcal{S}_0$, which is a motivation for this paper. In the context of hyperplane arrangements, this characterization was proved in \cite[Theorem 1.2]{Le} only for simple matroids. However, it holds true a little bit more generally:

\begin{lemma}\label{lm26}
 Let $M$ be a loopless matroid. Suppose that $(h_0(M),h_1(M),\ldots,h_s(M))$ is the $h$-vector of $BC(M)$ with $s$ equal to the largest index $i$ such that $h_i(M)\ne0$. Then the following conditions are equivalent:
\begin{enumerate}
 \item
 $M\in \mathcal{S}_0$;

\item
$(h_0(M),h_1(M),\ldots,h_s(M))$ is symmetric, i.e. $h_i(M)=h_{s-i}(M)$ for $i=0,\ldots,s$;

\item
there exists an ordering $<$ of the ground set of $M$ such that the broken circuit complex $BC(M,<)$ is a complete intersection;

 \item
each connected component of $M$ is either a coloop or an iterated parallel connection of non-loop circuits.
\end{enumerate}
\end{lemma}

\begin{proof}
 Denote by $\overline{M}$ the simplification of $M$. Let (i)', (ii)', and (iii)' be the conditions obtained from (i), (ii), and (iii) by replacing $M$ with $\overline{M}$, respectively. Also, let (iv)' be the condition: each connected component of $\overline{M}$ is either a coloop or an iterated parallel connection of simple circuits (i.e. circuits other than 2-circuits). Then the conditions (i)'--(iv)' are equivalent by \cite[Theorem 1.2]{Le}. Since (i), (ii), and (iii) are conditions on broken circuit complexes, \cite[Proposition 7.4.1]{B} ensures that (i)$\Leftrightarrow$(i)',  (ii)$\Leftrightarrow$(ii)', and (iii)$\Leftrightarrow$(iii)'. Hence, we have (i)$\Leftrightarrow$(ii)$\Leftrightarrow$(iii). Now to complete the proof it suffices to show (iv)'$\Rightarrow$(iv) and (iv)$\Rightarrow$(ii). Note that each connected component of $M$ is either a connected component of $\overline{M}$ or an iterated parallel connection of a connected component of $\overline{M}$ with 2-circuits. So (iv)' implies (iv).
Finally, assume (iv). Then the symmetry of the $h$-vector $(h_0(M),h_1(M),\ldots,h_s(M))$ follows easily from Lemma \ref{lm25}(ii), (iv), (v).
\end{proof}

\section{Contracting series classes}

Let $M$ be a matroid in $\mathcal{S}$. Lemma \ref{lm25}(vii) indicates that the computation of $\delta_1(M)$ reduces to the case where $M$ is connected, i.e. $M$ is a series-parallel network (in fact, if necessary, one may even assume that $M$ is parallel irreducible). As preparation for the next section where such a computation is carried out, we discuss in this section the variation of the number $\delta_1(M)$ when contracting $M$ by a subset of a series class.

Let $M$ be a series-parallel network of rank $r\geq2$ on the ground set $E$ and let $e\in E$. Note that $r(M-e)=r(M)=r$ and $r(M/e)=r-1$, so we may write the $h$-polynomials of broken circuit complexes of these matroids as follows
\[\begin{aligned}
h(M;x)&=h_0x^r+h_1x^{r-1}+\cdots+h_{r-2}x^2+h_{r-1}x,\\
h(M-e;x)&=h_0'x^r+h_1'x^{r-1}+\cdots+h_{r-2}'x^2+h'_{r-1}x,\\
h(M/e;x)&=h_0''x^{r-1}+h_1''x^{r-2}+\cdots+h_{r-3}''x^2+h''_{r-2}x.
\end{aligned}\]
The formula $h(M;x)=h(M-e;x)+h(M/e;x)$ in Lemma \ref{lm25}(iv) now gives $h_i=h_i'+h_{i-1}''$ for $i=1,\ldots,r-1$.
\begin{proposition}\label{pr32}
We keep the notation as above. Assume that $M-e$ is disconnected. Then the following statements hold.
\begin{enumerate}
\item
$M/e$ is a series-parallel network.

\item
$h_1''\leq h_1$, with equality if and only if $e$ is contained in no 3-circuit of $M$.

\item
If $r\geq 3$, then $h_{r-2}-1\leq h''_{r-3}\leq h_{r-2}$. The first inequality becomes an equality when $M-e$ has 2 connected components, and the second inequality becomes an equality when $M-e$ has at least 3 connected components.
\end{enumerate}
\end{proposition}

\begin{proof}
(i) Note that $M/e$ is a connected minor of $M$ by Lemma \ref{lm25}(iv). So according to Lemma \ref{lm24}(ii), $M/e$ is a series-parallel network.

(ii) Let $\overline{M}$ be the simplification of $M$. Recall that the broken circuit complexes of $M$ and $\overline{M}$ share the same $h$-vector; see \cite[Proposition 7.4.1]{B}. We use the notation $E(.)$ to denote the ground set of the matroid within the brackets. Since $M-e$ is disconnected, $e$ is contained in no 2-circuit of $M$, whence $e\in E(\overline{M})$. It is clear that $\overline{M/e}=\overline{\overline{M}/e}$. Accordingly, $E(\overline{M/e})\subseteq E(\overline{M}/e)$ with equality if and only if $e$ is contained in no 3-circuit of $\overline{M}$, i.e. if and only if $e$ is contained in no 3-circuit of $M$. Note that $r(\overline{M/e})=r(\overline{\overline{M}/e})=r(\overline{M}/e)$. So by Lemma \ref{lm25}(i), we obtain
\[\begin{aligned}
  h''_1&=|E(\overline{M/e})|-r(\overline{M/e})\leq |E(\overline{M}/e)|-r(\overline{M}/e)\\
&= (|E(\overline{M})|-1)-(r(\overline{M})-1)=|E(\overline{M})|-r(\overline{M})=h_1.
\end{aligned}\]
The equality holds if and only if $e$ is contained in no 3-circuit of $M$.

(iii) We have $h_{r-2}=h'_{r-2}+h''_{r-3}\geq h''_{r-3}$. The equality holds if and only if $h'_{r-2}=0$, which by Lemma \ref{lm25}(iii) means that $M-e$ has at least 3 connected components.

It remains to prove that $h_{r-2}'\leq 1$, with equality if and only if  $M-e$ has exactly $2$ connected components. We have seen that $h'_{r-2}=0$ when $M-e$ has at least 3 connected components. Now suppose $M-e$ has exactly $2$ connected components: $M-e=N_1\oplus N_2$. We need to show that $h_{r-2}'=1$. Indeed, it follows from Lemma \ref{lm24}(ii) that $N_1$ and $N_2$, which are connected minors of $M$, are series-parallel networks. So by Lemmas \ref{lm25}(ii) and \ref{lm24}(ii), $h_{r-2}'=\beta(N_1)\beta(N_2)=1$.
\end{proof}

\begin{corollary}\label{lm33}
 Let $M$ be a series-parallel network of rank $r\geq2$ on $E$. Let $e\in E$ be such that $M-e$ is disconnected. Then $M/e$ is a series-parallel network and $\delta_1(M/e)\geq \delta_1(M)-1$. Assume further that $e$ is contained in no 3-circuit of $M$. Then
\begin{enumerate}
 \item
$\delta_1(M/e)= \delta_1(M)-1$ if and only if $M-e$ has 2 connected components;

\item
$\delta_1(M/e)\leq \delta_1(M)$ with equality if and only if $M-e$ has at least 3 connected components.
\end{enumerate}
\end{corollary}

\begin{proof}
The fact that $M/e$ is a series-parallel network was proved in Proposition \ref{pr32}. Let us show that $\delta_1(M/e)\geq \delta_1(M)-1$. If $r\leq 3$, then it follows from the definition that $\delta_1(M/e)= \delta_1(M)=0$. If $r\geq4$, then $\delta_1(M)=h_{r-2}-h_1$ and $\delta_1(M/e)=h_{r-3}''-h_1''$. So the inequality follows easily from Proposition \ref{pr32}.

Now assume that $e$ is contained in no 3-circuit of $M$. Then $r\geq 3$. For the same reason as above, (i) and (ii) follow from Proposition \ref{pr32} when $r\geq 4$. Consider the case $r=3$. Then $\delta_1(M/e)= \delta_1(M)=0$. Thus to complete the proof, we need to show in this case that $M-e$ has 3 connected components. Let $\overline{M}$ be the simplification of $M$. Observe that there are only two non-isomorphic simple matroids of rank 3 which are series-parallel networks, namely, a 4-circuit and a parallel connection of two 3-circuits. Since $e$ is contained in no 3-circuit of $M$, $\overline{M}$ must be a 4-circuit. It then follows without difficulty that $M-e$ has 3 connected components.
\end{proof}

A subset $X$ of the ground set of a connected matroid $M$ is called \emph{removable} if $M-X$ is connected. We are interested in removable series classes because of the following corollary.

\begin{corollary}\label{co33}
 Let $M$ be a series-parallel network on $E$. Let $X$ be a subset of a series class of $M$. Assume that $|X|\geq2$ and $X\cup \{f\}$ is not a circuit of $M$ for every $f\in E$. Denote by $\tilde{M}$ the matroid obtained from $M$ by contracting all but one element of $X$. Then
\[\delta_1(M)=
\begin{cases}
 \delta_1(\tilde{M}) +1& \text{ if } X \text{ is removable},\\
\delta_1(\tilde{M}) & \text{otherwise}.
\end{cases}
\]
\end{corollary}

\begin{proof}
  Let $\hat{M}$ be $M$ with all but two elements, say $e_1$ and $e_2$, of $X$ contracted. We will show that  $\delta_1(M)=\delta_1(\hat{M})$, and that $\{e_1,e_2\}$ is a subset of a series class of $\hat{M}$. The case $|X|=2$ is trivial, so we may assume $|X|\geq 3$. Let $e\in X-\{e_1,e_2\}$. Since $X$ is a subset of a series class, every circuit of $M$ containing $e$ must contain all of $X$. It follows that $e$ is contained in no 3-circuit of $M$. Moreover, $M-e$ has at least 3 connected components because every element of $X-e$ is a coloop, and hence, a connected component of $M-e$. So by Corollary \ref{lm33}(ii), $\delta_1(M)=\delta_1(M/e)$. Clearly, $X-e$ is a subset of a series class of $M/e$. Therefore, we may repeat the above argument for $M/e$, and so on. Eventually, we obtain $\delta_1(M)=\delta_1(\hat{M})$, and $\{e_1,e_2\}$ is a subset of a series class of $\hat{M}$. From the assumption on $X$ it follows that  $\{e_1,e_2\} \cup\{f\}$ is not a circuit of $\hat{M}$ for every $f\in E(\hat{M})$.
Thus $e_2$ is not contained in any 3-circuit of $\hat{M}$. Observe that
\[\hat{M}-e_2=M/(X-\{e_1,e_2\})-e_2=(M-e_2)/(X-\{e_1,e_2\})=\{e_1\}\oplus(M-X).\]
So if $X$ is removable then $\hat{M}-e_2$ has exactly 2 connected components, and it follows from Corollary \ref{lm33}(i) that $\delta_1(\hat{M})=\delta_1(\hat{M}/e_2)+1=\delta_1(\tilde{M})+1$. Otherwise, $\hat{M}-e_2$ has at least 3 connected components and $\delta_1(\hat{M})=\delta_1(\hat{M}/e_2)=\delta_1(\tilde{M})$ by Corollary \ref{lm33}(ii).
\end{proof}

We conclude this section with a description of removable series classes of graphic matroids. Let $G$ be a graph. A path of $G$ is called a \emph{line} if all of its internal vertices have degree 2. When $G$ is a block, as for matroids, a line $X$ of $G$ is \emph{removable} if $G-X$ is a block. Evidently, every line of a block $G\ne K_2$ is a subset of a series class of $M(G)$. The converse is not true in general; see, e.g. \cite[p. 155]{O}. However, we have

\begin{proposition}\label{pr34}
  Let $G$ be a block. Then removable series classes of $M(G)$ are exactly removable lines of $G$.
\end{proposition}

\begin{proof}
 We first observe that if $X\ne\emptyset$ is a subset of a series class of a connected matroid $M$ such that $X$ is removable, then $X$ is a removable series class of $M$. From this it follows that removable lines of $G$ are removable series classes of $M(G)$. Conversely, assume $X$ is a removable series class of $M(G)$. Let $X'$ be a maximal line of $G$ contained in $X$. We will show that $X=X'$. By the observation at the beginning of the proof, it suffices to prove that $G-X'$ is a block. Let $u,v$ be the end vertices of $X'$. From the maximality of $X'$ it follows that $u$ and $v$ have degree at least 3. Let $e_1,e_2$ be two edges of $G-X'$ incident to $u$. Then any cycle of $G$ containing $e_1,e_2$ does not contain $X'$. This implies $e_1,e_2\not\in X$. So $u$ is a vertex of $G-X$. Similarly, $v$ is also a vertex of $G-X$. Since $G-X$ is a block, there exists a cycle $C\subseteq G-X$ connecting $u$ and $v$; see, e.g. \cite[Proposition 4.1.4]{O}. Thus the line $X'$ shares only its end vertices with the 
cycle $C$. It follows that $G-X'$ is a block. Hence $X=X'$ is a removable line of $G$.
\end{proof}

\section{Ear decompositions and $\delta_1$}

In this section we will extract some useful information from ear decompositions of series-parallel networks. In particular, a formula for the number $\delta_1$ of a series-parallel network will be derived. Several bounds for $\delta_1$ will then also be discussed.

An \emph{ear decomposition} of a graph $G$ is a partition of the edges of $G$ into a sequence of \emph{ears} $\pi_1,\pi_2,\ldots,\pi_n$ such that
\begin{enumerate}[(ED1)]
 \item
$\pi_1$ is a non-loop cycle and each $\pi_i$ is a simple path (i.e. a path that does not intersect itself) for $i\geq 2$;

\item
each end vertex of $\pi_i$, $i\geq 2$, is contained in some $\pi_j$ with $j<i$;

\item
no internal vertex of $\pi_i$ is in $\pi_j$ for any $j<i$.
\end{enumerate}
Whitney \cite{Wh2} proved that a graph with at least 2 edges admits an ear decomposition if and only if it is a block. He also showed that for a block $G=(V,E)$, the number $n$ of ears in an ear decomposition of $G$ is its \emph{nullity} (or \emph{cyclomatic number}), i.e. $n=|E|-|V|+1$. Thus, in particular, we see from Lemma \ref{lm25}(i) that $n=h_1(M(G))$ if $G$ is a simple graph.

One may characterize series-parallel networks by their ear decompositions. Given an ear decomposition of a graph $G$ as above, we say that $\pi_i$ is \emph{nested} in $\pi_j$, $j<i$, if both end vertices of $\pi_i$ belong to $\pi_j$ and at least one of them is an internal vertex of $\pi_j$. (Note that every vertex of the cycle $\pi_1$ is internal.) When $\pi_i$ is nested in $\pi_j$, the \emph{nest interval} of $\pi_i$ in $\pi_j$ is the path in $\pi_j$ between the two end vertices of $\pi_i$. Here we adopt the convention that the nest interval of an ear in $\pi_1$ is the shorter path, and that if $\pi_1$ is divided into two paths of equal length then at most one of them could be a nest interval. The given ear decomposition is called \emph{nested} if the following conditions hold:
\begin{enumerate}[(N1)]
 \item
for each $i>1$ there exists $j<i$ such that $\pi_i$ is nested in $\pi_j$;

\item
if $\pi_i$ and $\pi_k$ are both nested in $\pi_j$, then either their nest intervals in $\pi_j$ are disjoint, or one nest interval contains the other.
\end{enumerate}
In \cite{Ep}, it is proved that series-parallel networks with at least 2 edges are exactly those blocks for which some ear decomposition is nested, a condition which is equivalent to every ear decomposition being nested.

We have seen that the number of ears in an ear decomposition of a graph is an invariant of the graph and hence independent of the decomposition. Naturally, one may ask whether this is the case for the number of nest intervals in a nested ear decomposition. It turns out that this question is not as simple as it seems. In fact, we can only achieve an affirmative answer to the question through a subtle analysis of the set of nest intervals.

 Let $\Pi=(\pi_1,\pi_2,\ldots,\pi_n)$ be a nested ear decomposition of a series-parallel network $G$. For convenience, we will also view $\Pi$ as the set $\{\pi_1,\pi_2,\ldots,\pi_n\}$. Denote by $N(\Pi)$ the set of all nest intervals appearing in $\Pi$. For each nest interval $I\in N(\Pi)$, we consider the following collection of paths
\[\sigma(I) =\{\pi_i\in \Pi:I \text{ is the nest interval of } \pi_i\}.\]
By conditions (ED3) and (N1), each ear $\pi_i$, $i>1$, is nested in a unique other ear. It follows that the sets $\sigma(I)$, $I\in N(\Pi)$, are well-defined and constitute a partition of $\Pi-\pi_1$. Set $\sigma(I)^+=\sigma(I)\cup\{I\}$ and let $\ell(I)$ be the minimal length of a path in $\sigma(I)^+$ (recall that the \emph{length} of a path is its number of edges). Denote by $p_1(G;\Pi)$ and $p_2(G;\Pi)$ the number of $I\in N(\Pi)$ such that $\ell(I)=1$ and $\ell(I)>1$, respectively. We will see that both $p_1(G;\Pi)$ and $p_2(G;\Pi)$ not only are independent of the decomposition $\Pi$ but also encode interesting combinatorial information about the graph $G$. Thus, in particular, the cardinality of $N(\Pi)$ is an invariant of $G$.

\begin{example}\label{ex42}
 Let $G$ be the graph shown in Figure \ref{fig2}. Let $\pi_1=\{1,2,3,4,5\}$, $\pi_2=\{6\}$, $\pi_3=\{7\}$, $\pi_4=\{8,9,10\}$, $\pi_5=\{11,12\}$. Then $\Pi=(\pi_1,\ldots,\pi_5)$ is a nested ear decomposition of $G$. This decomposition has 3 nest intervals: $I_1=\{3\}$, $I_2=\{4,5\}$, $I_3=\{9,10\}$. Since $\ell(I_1)=\ell(I_2)=1$ and $\ell(I_3)=2$, we have $p_1(G;\Pi)=2$ and $p_2(G;\Pi)=1$.
\end{example}

\begin{figure}[htb]
 \begin{tikzpicture}[scale=1.5]
\draw (0,2)-- (0,0);
\draw (2.18,0)-- (0,0);
\draw (0,2)-- (2.18,2);
\draw (2.18,2)-- (2.18,0);
\draw (2.18,2)-- (1.18,1.08);
\draw (1.18,1.08)-- (2.18,0);
\draw (2.18,2)-- (4.06,1.16);
\draw (4.06,1.16)-- (3.36,0);
\draw (2.18,0)-- (3.36,0);
\draw (2.18,0)-- (4.06,1.16);
\draw (2.18,2) to [out=150,in=30] (0,2);
\draw (1.06,-0.06) node[anchor=north west] {1};
\draw (-0.3,1.22) node[anchor=north west] {2};
\draw (0.9,1.96) node[anchor=north west] {3};
\draw (1.5,1.52) node[anchor=north west] {4};
\draw (1.62,0.9) node[anchor=north west] {5};
\draw (1,2.7) node[anchor=north west] {6};
\draw (2.18,1.1) node[anchor=north west] {7};
\draw (3.06,1.94) node[anchor=north west] {8};
\draw (3.44,1.22) node[anchor=north west] {9};
\draw (2.5,0.76) node[anchor=north west] {10};
\draw (3.76,0.62) node[anchor=north west] {11};
\draw (2.7,-0.06) node[anchor=north west] {12};
\draw (2,-0.32) node[anchor=north west] {$G$};

\begin{scriptsize}
\fill [color=black] (0,0) circle (1.5pt);
\fill [color=black] (2.18,0) circle (1.5pt);
\fill [color=black] (0,2) circle (1.5pt);
\fill [color=black] (2.18,2) circle (1.5pt);
\fill [color=black] (1.18,1.08) circle (1.5pt);
\fill [color=black] (4.06,1.16) circle (1.5pt);
\fill [color=black] (3.36,0) circle (1.5pt);
\fill [color=black] (3.21,0.64) circle (1.5pt);
\end{scriptsize}
 \end{tikzpicture}
\caption{A series-parallel network.}\label{fig2}
\end{figure}
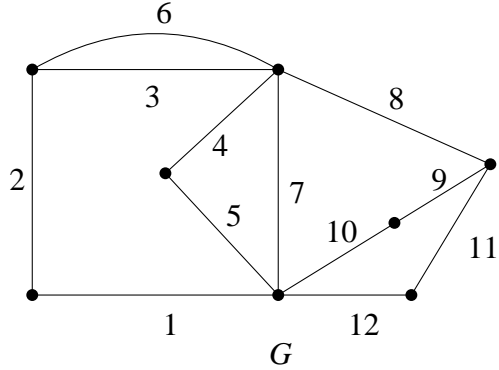

We first show that $p_1(G;\Pi)$ is independent of the decomposition $\Pi$.

\begin{theorem}\label{th42}
 Let $G$ be a series-parallel network with at least 2 edges. Denote by $F(G)$ the set of all edges $e$ of $G$ such that $G/e$ is not a block. Let $\overline{F}(G)=F(G)\cap E(\overline{G})$, where $ E(\overline{G})$ is the edge set of the simplification $\overline{G}$ of $G$. Then for any nested ear decomposition $\Pi$ of $G$ we have $p_1(G;\Pi)=|\overline{F}(G)|$. In particular, if $G$ is simple, then $p_1(G;\Pi)$ is the number of distinct baseedges of the parallel irreducible decomposition of $G$.
\end{theorem}

\begin{remark}\label{rm43}
Let $\Pi=(\pi_1,\pi_2,\ldots,\pi_n)$ be a nested ear decomposition of $G$. When dealing with  $p_1(G;\Pi)$ and $p_2(G;\Pi)$, it is possible to assume $|I|=1$ for any nest interval $I\in N(\Pi)$ with $\ell(I)=1$. Indeed, if $|I|>1$ then there exists an ear $\pi_i\in\sigma(I)$ such that $|\pi_i|=1$. So we can ``interchange'' $I$ and $\pi_i$ to get a new ear decomposition with the desired property and without changing $p_1(G;\Pi)$ and $p_2(G;\Pi)$. More precisely, let $\pi_j$ be the ear of $\Pi$ containing $I$ and consider the ear decomposition $\Pi'=(\pi_1',\pi_2',\ldots,\pi_n')$ of $G$ in which
\[\pi_j'=(\pi_j-I)\cup \pi_i,\quad \pi_i'=I,\quad \pi_k'=\pi_k\ \text{ for }\ k\ne i,j. \]
Then $I':=\pi_i$ becomes a nest interval in $N(\Pi')$ and $\ell(I')=|I'|=1$. Moreover, it is clear that $p_i(G;\Pi)=p_i(G;\Pi')$ for $i=1,2$.
\end{remark}

As an illustration of the above remark, let us revisit the graph $G$ in Figure \ref{fig2}. The ear decomposition $\Pi$ of $G$ given in Example \ref{ex42} has a nest interval $I_2$ with $\ell(I_2)=1<|I_2|$. By applying the procedure in Remark \ref{rm43} to $\Pi$ we obtain the nested ear decomposition $\Pi'=(\pi_1',\pi_2,\pi_3',\pi_4,\pi_5)$, where $\pi_1'=\{1,2,3,7\}$ and $\pi_3'=\{4,5\}$.

In order to prove Theorem \ref{th42} (and several results below), we introduce here a special kind of nest interval which is useful for inductive arguments. Recall that a line of $G$ is a path all of whose internal vertices have degree 2. Let $\Pi$ be a nested ear decomposition of $G$. We say that a nest interval $I\in N(\Pi)$ is \emph{lined} in $G$ if all paths of $\sigma(I)^+$ are lines in $G$. The existence and some properties of this kind of nest interval are shown below.

\begin{lemma}\label{lm46}
 Let $\Pi$ be a nested ear decomposition of a series-parallel network $G$. If $N(\Pi)\ne\emptyset$, then there exists $I\in N(\Pi)$ which is lined in $G$.
\end{lemma}

\begin{proof}
We prove by induction on $p=|N(\Pi)|$. If $p=1$, then it is clear that the unique nest interval of $N(\Pi)$ is lined in $G$. Assume $p\geq 2$.
Let $I$ be a minimal element (with respect to inclusion) of $N(\Pi)$. Then $I$ is a line of $G$. So if all the ears in $\sigma(I)$ are lines, we are done. Otherwise, let $\pi_j\in \sigma(I)$ be not a line. We say that an ear $\pi_i$ is \emph{sequentially nested} in $\pi_j$ if there exists a sequence of ears $\pi_i,\pi_k,\ldots,\pi_j$ such that each ear is nested in the next. Denote by $\Pi(\pi_j)$ the subset of $\Pi$ consisting of all ears which are sequentially nested in $\pi_j$. Note that $\overline{\pi}_j=\pi_j\cup I$ is a cycle. So one may check that $\Pi_1=\{\overline{\pi}_j\}\cup\Pi(\pi_j)$ is a nested ear decomposition of the subgraph $G_1$ of $G$ induced by the ears in $\Pi_1$. We have $N(\Pi_1)\subseteq N(\Pi)-I$. Moreover, $N(\Pi_1)\ne\emptyset$ since $\pi_j$ is not a line. Hence, by induction there exists $I_1\in N(\Pi_1)$ that is lined in $G_1$. Clearly, $I_1$ is also lined in $G$.
\end{proof}

\begin{lemma}\label{lm47}
Let $\Pi$ be a nested ear decomposition of a series-parallel network $G$. Assume that $I\in N(\Pi)$ is lined in $G$. Let $G'$ be the subgraph of $G$ induced by the ears in $\Pi':=\Pi-\sigma(I)$. Then the following statements hold.
\begin{enumerate}
 \item
$\Pi'$ is a nested ear decomposition of $G'$ and $N(\Pi')=N(\Pi)-I$.

\item
$I$ is a removable series class of the cycle matroid $M(G)$.

\item
$I$ is a subset of a series class of $M(G')$. Moreover, $I$ is not removable in $M(G')$ unless $G'$ is a $2$-cycle.
\end{enumerate}
\end{lemma}

\begin{proof}
 (i) Since all ears of  $\sigma(I)$ are lines, there is no ear of $\Pi$ which is nested in any ear of $\sigma(I)$. It follows that $\Pi'$ is a nested ear decomposition of $G'$. Obviously, $N(\Pi')=N(\Pi)-I$.

(ii) It is clear that $I$ is a removable line of $G$. So by Proposition \ref{pr34}, $I$ is a removable series class of $M(G)$.

(iii) Since $I$ is a line of $G'$, it is a subset of a series class of $M(G')$. Let us prove that $I$ is not removable in $M(G')$ when $G'$ is not a $2$-cycle. Suppose $I$ lies on the ear $\pi_j$ of $\Pi'$ and $\Pi'=(\pi_1,\ldots,\pi_j,\ldots,\pi_i)$. Let $e$ be an edge of $I$. Denote by $\tilde{G}'$ the contraction of $G'$ by all the edges of $I$ but $e$, i.e. $\tilde{G}'=G'/(I-e)$. Then $G'/I=\tilde{G}'/e$, and since $I$ is a line, $G'-I=\tilde{G}'-e$.  It is easily seen that
\[\Pi'/(I-e):=(\pi_1,\ldots,\pi_j/(I-e),\ldots,\pi_i)\]
 is a nested ear decomposition of $\tilde{G}'$, where $\pi_j/(I-e)$ denotes the contraction of $\pi_j$ by $I-e$. Thus we have: (a) $\tilde{G}'$ is a series-parallel network. We will further show that: (b) $\tilde{G}'$ has at least 3 edges, and (c) $G'/I$ is a series-parallel network with a nested ear decomposition $ \Pi'/I:=(\pi_1,\ldots,\pi_j/I,\ldots,\pi_i)$. First, assume $j>1$. Then (b) is clear from the ear decomposition of $\tilde{G}'$ given above. (c) is also clear because $I$ is properly contained in $\pi_j$ (by definition) and $I$ is not a nest interval of any ear in $\Pi'$. Next, consider the case $j=1$. By convention, the length of the path $\pi_1-I$ is at least that of $I$. If $\pi_1$ is a 2-cycle, then $\sigma(I)=\Pi-\pi_1$ since $I$ is lined in $G$. Consequently, $G'$ would be $\pi_1$ and hence a 2-cycle, contradicting the assumption. So $\pi_1$ has at least 3 edges, whence $\pi_1/I$ is a non-loop cycle. It follows that (b) and (c) are also true in this case. Now from (a), (b), (c), and Lemma
\ref{lm24}(ii) we deduce that $G'-I=\tilde{G}'-e$ is not a block, i.e. $I$ is not removable in $M(G')$.
\end{proof}

We are now ready to prove Theorem \ref{th42}.

\begin{proof}[Proof of Theorem \ref{th42}]
We first prove the formula $p_1(G;\Pi)=|\overline{F}(G)|$ for every ear decomposition $\Pi$ of $G$. We argue by induction on the nullity $n$ of $G.$ If $n=1$, then $G$ is a cycle. In this case, $G$ has a unique ear decomposition $\Pi=(G)$ and $p_1(G;\Pi)=|\overline{F}(G)|=0$. Now suppose that $n>1$. Let $\Pi$ be an arbitrary ear decomposition of $G$. By Lemma \ref{lm46}, there exists a nest interval $I\in N(\Pi)$ which is lined in $G$. Let $\sigma(I)=\{\pi_{i_1},\ldots,\pi_{i_s}\}$ and $\Pi'=\Pi-\sigma(I)$.  Then by Lemma \ref{lm47}(i), $\Pi'$ is a nested ear decomposition of the subgraph $G'$ of $G$ induced by the ears in $\Pi'$ and
\begin{equation}\label{eq1}
 N(\Pi)=N(\Pi')\sqcup \{I\},
\end{equation}
where $\sqcup$ denotes disjoint union. We distinguish two cases:

Case 1: $\ell(I)=1$. By Remark \ref{rm43}, we may assume $I$ contains only one edge $e$ of $G$. Then $G=P(G',D_1,\ldots,D_s)$ with respect to the baseedge $e$, where $D_t=I\cup \pi_{i_t}$ is a cycle for $t=1,\ldots,s$. We will show that
\begin{equation}\label{eq2}
 F(G)=F(G')\sqcup [e],
\end{equation}
where $[e]$ is the parallel class of $e$ in $G$. First, we check that $F(G')\cap [e]=\emptyset$. When $G'$ is a 2-cycle, this is true since $F(G')=\emptyset$. Assume $G'$ has at least 3 edges. Then $e$ is the only edge of $[e]$ which belongs to $G'$ (since all other edges of $[e]$ are ears in $\sigma(I)$). From the proof of Lemma \ref{lm47}(iii) we have $G'/e=G'/I$ is a block. So $e\not\in F(G')$, and therefore, $F(G')\cap [e]=\emptyset$. It remains to prove that $F(G)=F(G')\cup [e]$. By Lemma \ref{lm22}(iii), $e\in F(G)$. This implies $[e]\subseteq F(G)$. Next, let $f$ be an edge of $G-[e]$. If $f$ is in $G'$, then $G'$ has at least 3 edges (otherwise $G'$ would be a 2-cycle and $f\in [e]$). By Lemma \ref{lm22}(iii), $G/f=P(G'/f,D_1,\ldots,D_s).$
 Hence, it follows from Lemma \ref{lm22}(ii) that $f\in F(G')$ if and only if $f\in F(G)$. On the other hand, if $f$ belongs to some cycle $D_t$, then $D_t$ also has at least 3 edges. So by Lemma \ref{lm22}(ii), (iii),
$G/f=P(G',D_1,\ldots,D_t/f,\ldots,D_s)$
is a block, i.e. $f\not\in F(G)$. We conclude that $F(G)=F(G')\cup [e]$ and thereby \eqref{eq2} holds. Now from \eqref{eq1}, \eqref{eq2}, and the induction hypothesis we obtain
\[p_1(G;\Pi)=p_1(G';\Pi')+1=|\overline{F}(G')|+1=|\overline{F}(G)|.\]

Case 2: $\ell(I)=k>1$. Then $p_1(G;\Pi)=p_1(G';\Pi')$. So by the induction hypothesis, it suffices to prove that $F(G)=F(G')$. Let $e$ be an edge of $I$. Denote by $\tilde{G}$ the contraction graph $G/(I-e)$. Then since $I$ is a line, $G\cong S(\tilde{G},C_k)$, where $C_k$ is a $k$-cycle. We will show that $F(G)=F(\tilde{G})-e$. For any edge $f$ of $I$, $G-f$ is not a block since $I$ is a line. Hence by Lemma \ref{lm25}(iv), $G/f$ is a block, i.e. $f\not\in F(G)$. Now let $f$ be an edge of $G-I$. Then $f$ is also an edge of $\tilde{G}-e$. So by Lemma \ref{lm22}(iii), $G/f\cong S(\tilde{G}/f,C_k)$. Thus, it follows from Lemma \ref{lm22}(ii) that $f\in F(G)$ if and only if $f\in F(\tilde{G})$. Consequently, $F(G)=F(\tilde{G})-e$. By the same argument, we obtain $F(G')=F(\tilde{G}')-e$, in which $\tilde{G}'=G'/(I-e)$. Note that $\tilde{I}:=\{e\}$ is a nest
interval in $N(\tilde{\Pi})$, where $\tilde{\Pi}$ is the ear decomposition of $\tilde{G}$ induced by $\Pi$ ($\tilde{\Pi}=\Pi/(I-e)$ in the notation of the proof of Lemma \ref{lm47}(iii)). So as shown in Case 1, we have $F(\tilde{G})=F(\tilde{G}')\cup \tilde{[e]}$ with $\tilde{[e]}$ the parallel class of $e$ in $\tilde{G}$. Since $\ell(I)>1$, it is clear that $\tilde{[e]}=\{e\}$. Now by combining all the equalities just established, we get
\[F(G)=F(\tilde{G})-e=(F(\tilde{G}')\cup \tilde{[e]})-e=(F(\tilde{G}')\cup \{e\})-e=F(\tilde{G}')-e=F(G').\]

The first assertion of the theorem has been proved. For the second one, suppose $G$ is simple. Then $\overline{F}(G)=F(G)$. From the proof of Lemma \ref{lm22}(viii) (and Lemma \ref{lm22}(v)) we know that $F(G)$ is the set of baseedges of the parallel irreducible decomposition of $G$. Thus the second assertion of the theorem follows from the first one.
\end{proof}

\begin{remark}\label{rm46}
With the notation of Theorem \ref{th42}, one can always find a parallel irreducible decomposition of $G$ whose set of baseedges is $\overline{F}(G)$, even when $G$ is not simple. Indeed, let $e$ be an element of $\overline{F}(G)$ with a non-trivial parallel class $[e]=\{e,e_1,\ldots,e_k\}$. Then $G=P(G',D_1,\ldots,D_k)$ with respect to the baseedge $e$, where $G'=G-\{e_1,\ldots,e_k\}$ and $D_i=\{e,e_i\}$ is a 2-cycle for $i=1,\ldots,k$. Since $F(G')=F(G)-[e]$, we have $\overline{F}(G')=\overline{F}(G)-e$. Now a desired parallel irreducible decomposition of $G$ can be derived by induction.
\end{remark}

The corollary below, which follows at once from Theorem \ref{th42} and Lemma \ref{lm22}(vi), provides an interesting characterization of parallel irreducibility of series-parallel networks in terms of ear decompositions.

\begin{corollary}\label{lm43}
  Let $G$ be a series-parallel network with at least 2 edges. Then $G$ is parallel irreducible if and only if $p_1(G;\Pi)=0$ for an ear decomposition $\Pi$ of $G$.
\end{corollary}

We now come to the main result of the paper.

\begin{theorem}\label{th45}
 Let $G$ be a series-parallel network with at least 2 edges. Then for any ear decomposition $\Pi$ of $G$ we have $\delta_1(M(G))=p_2(G;\Pi)$, where $M(G)$ is the cycle matroid of $G$. In particular, $p_2(G;\Pi)$ does not depend on $\Pi$.
\end{theorem}

\begin{proof}
 We argue by induction on the nullity $n$ of $G$ as in the proof of Theorem \ref{th42}. If $n=1$, then $G$ is a cycle. So $G$ has a unique ear decomposition $\Pi=(G)$ and $p_2(G;\Pi)=0$. On the other hand, $\delta_1(M(G))=0$ by Lemma \ref{lm25}(v). Thus the formula holds in this case. Now assume $n>1$. For an ear decomposition $\Pi$ of $G$, let $I, \Pi',G'$ have the same meaning as in the proof of Theorem \ref{th42}. Consider the following cases:

Case 1: $\ell(I)=1$. Then $G=P(G',D_1,\ldots,D_s)$, where the $D_i$ are cycles. So it follows from Lemma \ref{lm25}(v), (vii) that
\[\delta_1(M(G))=\delta_1(M(G'))+\delta_1(M(D_1))+\cdots+\delta_1(M(D_s))=\delta_1(M(G')).\]
Since $\ell(I)=1$, $p_2(G;\Pi)=p_2(G';\Pi')$. Thus from the induction hypothesis we obtain
\[\delta_1(M(G))=\delta_1(M(G'))=p_2(G';\Pi')=p_2(G;\Pi).\]

Case 2: $\ell(I)>1$. In this case, we have $p_2(G;\Pi)=p_2(G';\Pi')+1$. So by the induction hypothesis, we only need to show that $\delta_1(M(G))=\delta_1(M(G'))+1$. Let $e$ be an edge of $I$. Let $\tilde{G}=G/(I-e)$ and $\tilde{G}'=G'/(I-e)$. Since $\ell(I)>1$, there is no edge $f$ of $G$ such that $I\cup \{f\}$ is a cycle. It also follows from $\ell(I)>1$ that $G'$ is not a 2-cycle. Thus by Corollary \ref{co33} and Lemma \ref{lm47}(ii), (iii),
\begin{equation}\label{eq5}
 \delta_1(M(G))=\delta_1(M(\tilde{G}))+1,\quad \delta_1(M(G'))=\delta_1(M(\tilde{G}')).
\end{equation}
Let $\tilde{\Pi}=\Pi/(I-e)$ be the ear decomposition of $\tilde{G}$ induced by $\Pi$. Then $\tilde{I}:=\{e\}$ is a nest interval in $N(\tilde{\Pi})$ with $\ell(\tilde{I})=1$. So by Case 1, $\delta_1(M(\tilde{G}))=\delta_1(M(\tilde{G}')).$ Combining the latter equality with \eqref{eq5} we obtain $\delta_1(M(G))=\delta_1(M(G'))+1$. The theorem has been proved.
\end{proof}

\begin{example}\label{ex49}
Let $G$ be the graph in Figure \ref{fig2}. Let $\Pi$ be the ear decomposition of $G$ considered in Example \ref{ex42}. Then $\delta_1(M(G))=p_2(G;\Pi)=1.$ Note that the $h$-polynomial of $BC(M(G))$ is $x^7+4x^6+9x^5+12x^4+10x^3+5x^2+x$.
\end{example}

\begin{remark}
 Theorem \ref{th45}, unfortunately, does not hold for general graphic matroids. For instance, consider the graph $H$ depicted in Figure \ref{fig5}. Clearly, $H$ has a $K_4$-minor and so it is not a series-parallel network. Let $\Pi$ be the ear decomposition: $\pi_1=\{1,2,3,4\},$ $\pi_2=\{5,7\},$ $\pi_3=\{6\},$ $\pi_4=\{8\}$. Since this decomposition is not nested, $p_2(H;\Pi)$ does not make sense. Even if we tried to extend the definition of $p_2$ to include this case, then from the lengths of the ears in $\Pi$ we would have $p_2(H;\Pi)\le1$. However, the $h$-polynomial of $BC(M(H))$ is $x^4+4x^3+6x^3+3x$, and so $\delta_1(M(H))=2$.
\end{remark}

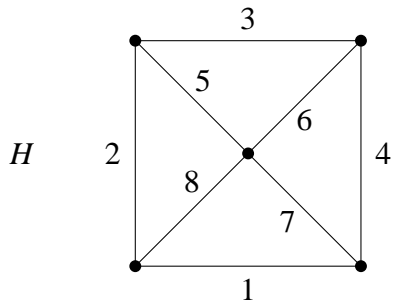
\begin{figure}[htb]
	\begin{tikzpicture}[scale=1.5]
	\draw (0,2)-- (0,0);
	\draw (2,0)-- (0,0);
	\draw (0,2)-- (2,2);
	\draw (2,2)-- (2,0);
	\draw (2,2)-- (1,1);
	\draw (1,1)-- (2,0);
	\draw (1,1)-- (0,0);
	\draw (1,1)-- (0,2);
	\draw (1,-0.2) node {1};
	\draw (-0.2,1) node {2};
	\draw (1,2.2) node {3};
	\draw (1.5,1.3) node {6};
	\draw (0.6,1.65) node  {5};
	\draw (1.35,0.4) node {7};
	\draw (0.5,0.75) node  {8};
	\draw (2.2,1) node {4};
	\draw (-1,1) node {$H$};

	\begin{scriptsize}
	\fill [color=black] (0,0) circle (1.5pt);
	\fill [color=black] (2,0) circle (1.5pt);
	\fill [color=black] (0,2) circle (1.5pt);
	\fill [color=black] (2,2) circle (1.5pt);
	\fill [color=black] (1,1) circle (1.5pt);
	\end{scriptsize}
	\end{tikzpicture}
	\caption{A non series-parallel network.}\label{fig5}
\end{figure}

From Theorems \ref{th42} and \ref{th45} we immediately get the following consequence which gives an affirmative answer to the question raised at the beginning of this section.

\begin{corollary}\label{co410}
 Let $G$ be a series-parallel network with at least 2 edges. Let $M(G)$ be the cycle matroid of $G$. Then the number $p$ of nest intervals appearing in a nested ear decomposition of $G$ is independent of the decomposition. Moreover, $\delta_1(M(G))\leq p$ with equality if and only if $M(G)$ is parallel irreducible.
\end{corollary}

The next corollary also follows easily from Theorem \ref{th45}.

\begin{corollary}\label{co411}
Let $G,G'$ be series-parallel networks such that $G$ is obtained from $G'$ by adding a new ear. Then $\delta_1(M(G'))\leq \delta_1(M(G))\leq \delta_1(M(G'))+1.$
\end{corollary}

As another application of Theorem \ref{th45}, we prove that the inequality in Corollary \ref{lm33}(ii) holds without the assumption that $e$ is contained in no 3-circuit of $M$.

\begin{corollary}\label{co412}
 Let $M$ be a series-parallel network of rank $r\geq2$ and $e$ an element in the ground set of $M$. If $M-e$ is disconnected, then $\delta_1(M/e)\leq \delta_1(M)$.
\end{corollary}

\begin{proof}
 Let $G$ be a series-parallel network with $M(G)=M$ and $\Pi$ an ear decomposition of $G$. By assumption, $G/e$ is also a series-parallel network. So it is easily seen that $\Pi/e$ is an ear decomposition of $G/e$. (For the notation $\Pi/e$, see the proof of Lemma \ref{lm47}(iii).) Since $p_2(G/e;\Pi/e)\leq p_2(G;\Pi)$, the corollary follows from Theorem \ref{th45}.
\end{proof}

In the remaining part of this section, we derive from Theorem \ref{th45} several bounds for the number $\delta_1(M)$ of a series-parallel network $M$. We first compare $\delta_1(M)$ with $h_1(M)$, the second entry of the $h$-vector of a broken circuit complex of $M$.

\begin{proposition}\label{pr411}
Let $M$ be a series-parallel network of rank $r\geq2$. Then
$\delta_1(M)\leq h_1(M)-1$, where $(h_0(M),h_1(M),\ldots,h_r(M))$ is the $h$-vector of $BC(M)$. If equality holds, then $M$ is parallel irreducible.
\end{proposition}

\begin{proof}
By \cite[Proposition 7.4.1]{B}, we may assume that $M$ is simple. Let $G$ be a block with $M(G)=M$. Recall from Lemma \ref{lm25}(i) that $h_1(M)=n$, where $n$ is the nullity of $G$. Let $\Pi$ be an ear decomposition of $G$. Denote by $p$ the number of nest intervals appearing in $\Pi$. Then it is obvious that $p\leq n-1$. Now Theorem \ref{th45} yields
\[\delta_1(M)=p_2(G;\Pi)\leq p\leq n-1=h_1(M)-1.\]
If $\delta_1(M)=h_1(M)-1$, then $\delta_1(M)= p$, which by Corollary \ref{co410} means that $M$ is parallel irreducible.
\end{proof}

With the notation of the previous proposition, one has $\delta_1(M)=h_{r-2}(M)-h_1(M)$ if $r\geq 3$. Therefore, the following corollary follows immediately from this proposition, Lemma \ref{lm25}(i), and Theorem \ref{th45}.

\begin{corollary}\label{co413}
Keep the notation of Proposition \ref{pr411}. Assume $r\geq 3.$ Let $G$ be a series-parallel network with $M(G)=M$ and let $\overline{G}$ be the simplification of $G$. Then for any ear decomposition $\Pi$ of $G$, $h_{r-2}(M)=p_2(G;\Pi)+n(\overline{G})$, where $n(\overline{G})$ is the nullity of $\overline{G}$. In particular, $h_{r-2}(M)\leq 2h_1(M)-1$ with equality only if $M$ is parallel irreducible.
\end{corollary}

The proof of Theorem \ref{th45} suggests that one may relate $\delta_1(M(G))$ to the number of vertices of $G$ of degree at least 3. The next result realizes this idea.

\begin{proposition}\label{pr414}
Let $M$ be a series-parallel network and $G$ a block with $M(G)=M$. Denote by $\nu(G)$ the number of vertices of $G$ of degree at least 3. Then the following statements hold.
\begin{enumerate}
\item
$\delta_1(M)\leq 2\nu(G)-3$ when $\nu(G)>0$. If the equality holds, then $M$ is parallel irreducible.

\item
Suppose $M$ is parallel irreducible. Let $\mu(G)$ be the number of pairs of vertices of $G$ which are connected by a removable line. Then $\delta_1(M)\geq \max\{\mu(G),\nu(G)/2\}$.
\end{enumerate}
\end{proposition}

\begin{proof}
 (i) Obviously, one has $\nu(G)\geq2$ as soon as $\nu(G)>0$. If $\nu(G)=2$, then every ear decomposition of $G$ has only one nest interval (which is a line connecting the two vertices of degree at least 3). Thus by Corollary \ref{co410}, $\delta_1(M)\leq 1= 2\nu(G)-3$ with equality if and only if $M$ is parallel irreducible. Now suppose that $\nu(G)>2$. Let $\Pi$ be an ear decomposition of $G$. Then by Lemma \ref{lm46}, there exists a nest interval $I\in N(\Pi)$ which is lined in $G$. Let $e$ be an edge of $I$. By Remark \ref{rm43}, we may assume that $I=\{e\}$ when $\ell(I)=1$. Denote by $G'$ the subgraph of $G$ induced by the ears in $\Pi':=\Pi-\sigma(I)$ and let $\tilde{G}'=G'/(I-e)$. Since $\nu(G')\geq \nu(G)-2>0$, $G'$ is not a 2-cycle. Hence by Lemma \ref{lm47}(iii), $M(\tilde{G}')-e=M(G')-I$ is disconnected. Note that $M(G'/I)=M(\tilde{G}')/e$ and $\nu(G)-2\leq\nu(G'/I)\leq\nu(G)-1$. So from Corollary \ref{lm33} and the induction hypothesis it follows that
\begin{equation}\label{eq4}
 \delta_1(M(\tilde{G}'))\leq \delta_1(M(G'/I))+1\leq 2\nu(G'/I)-3+1\leq 2\nu(G)-4.
\end{equation}
Since $G'=\tilde{G}'$ when $\ell(I)=1$, we obtain from \eqref{eq4} and the proof of Theorem \ref{th45} that
\[\begin{aligned}
\delta_1(M)
&=
  \begin{cases}
   \delta_1(M(G')) & \text{ if }\ \ell(I)=1,\\
   \delta_1(M(\tilde{G}'))+1& \text{ if }\ \ell(I)>1
  \end{cases}\\
&\leq
  \begin{cases}
   2\nu(G)-4 & \text{ if }\ \ell(I)=1,\\
   2\nu(G)-3& \text{ if }\ \ell(I)>1.
  \end{cases}
 \end{aligned}
\]

Let us now examine the case $\delta_1(M)= 2\nu(G)-3$. Then the argument above shows that $\ell(I)>1$ and $\delta_1(M(G'/I))= 2\nu(G'/I)-3$. The latter equality together with the induction hypothesis implies that $M(G'/I)$ is parallel irreducible. Let $\Pi'/I$ be the ear decomposition of $G'/I$ induced by $\Pi'$ (see the proof of Lemma \ref{lm47}(iii)). Observe that $N(\Pi'/I)=N(\Pi')$. So $N(\Pi)=N(\Pi'/I)\cup\{I\}$ by Lemma \ref{lm47}(i). Now using Corollary \ref{lm43} we conclude that $M$ is parallel irreducible.

(ii) We also argue by induction on $\nu(G)$ as in (i). If $\nu(G)=0$, then $G$ is a cycle. So $\mu(G)=0$, and $\delta_1(M)=0$ by Lemma \ref{lm25}(v). Assume that $\nu(G)\geq2$. Let $\Pi$, $I$, $\Pi'$, and $G'$ have the same meaning as in (i). Then $\nu(G')\geq\nu(G)-2$ and $N(\Pi')\subset N(\Pi)$. Since $M$ is parallel irreducible, it follows from Corollary \ref{lm43} that $\ell(I)>1$ and $M(G')$ is parallel irreducible. So by Theorem \ref{th45} and the induction hypothesis,
\[\delta_1(M)=\delta_1(M(G'))+1\geq \nu(G')/2+1\geq \nu(G)/2.\]
To complete the proof, it suffices to show that $\mu(G')\geq \mu(G)-1$. Let $u,v$ be a pair of vertices of $G$, other than the pair consisting of the end vertices of $I$, which are connected by a removable line $L$ of $G$.  Clearly, $u,v$ have degree at least 3. It follows that $u,v$ are vertices of $G'$ and $L$ is contained in $G'-I$. So we only need to prove that $L$ is a removable line of $G'$. By \cite[Proposition 4.1.4]{O}, this will follow once it is shown that every pair of distinct edges of $G'-L$ are contained in a cycle of $G'-L$. Let $e_1,e_2$ be distinct edges of $G'-L$. Since $G-L$ is a block, there is a cycle $C$ of $G-L$ containing both $e_1$ and $e_2$. If $C\subseteq G'-L$, we are done. Otherwise, $C$ must contain an ear $\pi_i\in \sigma(I)$. In this case, $C':=(C-\pi_i)\cup I$ is a cycle of $G'-L$ which contains both $e_1$ and $e_2$.
\end{proof}

\begin{remark}
We keep the notation of  Proposition \ref{pr414}.

(i) The bounds given in Proposition \ref{pr414} are tight. For instance, the cycle matroid of the complete bipartite graph $K_{2,3}$ attains all these bounds.

 (ii) By taking iterated parallel connection of cycles, one may construct a series-parallel network $G$ with $\delta_1(M(G))=0$, but $\mu(G)$ and $\nu(G)$ are arbitrarily large. This shows that the assumption of the parallel irreducibility of the matroid $M$ in Proposition \ref{pr414}(ii) is essential.

(iii) The number $\nu(G)$ is not an invariant of the matroid $M$, but depends on the graph $G$. For example, the graphs $G_1$ and $G_2$ shown in Figure \ref{fig3} have isomorphic cycle matroids, but $\nu(G_1)=4\ne 3=\nu(G_2)$. So we actually proved in Proposition \ref{pr414} that
\[\max\{\nu(G)/2: M=M(G)\}\leq\delta_1(M)\leq \min\{2\nu(G)-3: M=M(G)\},\]
where the first inequality holds under the assumption that $M$ is parallel irreducible.

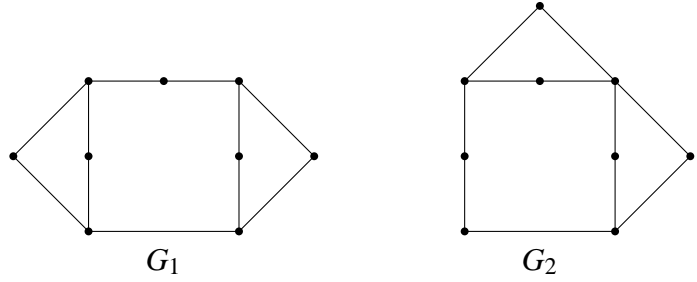
\begin{figure}[htb]
\begin{tikzpicture}
\draw (0,2)-- (0,0);
\draw  (0,0)-- (2,0);
\draw  (2,0)-- (2,2);
\draw (2,2)-- (0,2);
\draw (0,2)-- (-1,1);
\draw (-1,1)-- (0,0);
\draw (2,2)-- (3,1);
\draw (3,1)-- (2,0);
\draw (1,-0.4) node {$G_1$};
\draw (5,0)-- (7,0);
\draw  (7,0)-- (7,2);
\draw (7,2)-- (5,2);
\draw  (5,2)-- (5,0);
\draw (6,3)-- (5,2);
\draw (6,3)-- (7,2);
\draw (7,2)-- (8,1);
\draw (7,0)-- (8,1);
\draw (6,-0.4) node {$G_2$};
\begin{scriptsize}
\fill  (0,2) circle (1.5pt);
\fill  (0,1) circle (1.5pt);
\fill (0,0) circle (1.5pt);
\fill  (2,0) circle (1.5pt);
\fill  (2,2) circle (1.5pt);
\fill  (1,2) circle (1.5pt);
\fill  (2,1) circle (1.5pt);
\fill  (-1,1) circle (1.5pt);
\fill  (3,1) circle (1.5pt);
\fill  (5,0) circle (1.5pt);
\fill  (5,1) circle (1.5pt);
\fill  (7,0) circle (1.5pt);
\fill  (6,2) circle (1.5pt);
\fill (7,2) circle (1.5pt);
\fill  (7,1) circle (1.5pt);
\fill  (5,2) circle (1.5pt);
\fill (6,3) circle (1.5pt);
\fill (8,1) circle (1.5pt);
\end{scriptsize}
\end{tikzpicture}
\caption{$M(G_1)\cong M(G_2)$ but $\nu(G_1)\ne \nu(G_2)$.}\label{fig3}
\end{figure}

Meanwhile, $\mu(G)$ depends only on the matroid $M$, but not the graph $G$. Indeed, recall from Proposition \ref{pr34} that removable lines of $G$ are exactly removable series classes of $M$. Now on the set $\mathcal{R}$ of removable series classes of $M$ we define a relation $\sim$ as follows: $X_1\sim X_2$ if and only if either $X_1=X_2$ or $X_1\cup X_2$ is a circuit of $M$. Then it is easy to see that $\sim$ is an equivalence relation and $\mu(G)$ is equal to the cardinality of the quotient set $\mathcal{R}/\sim$. Thus, in particular, $\mu(G)$ is independent of $G$.
\end{remark}

\section{Applications}

Further applications of Theorem \ref{th45} will be derived in this section. Among them is an excluded minor characterization of the class $\mathcal{S}_0$. This result will then be used to examine outerplanar graphs and $A$-graphs.

We first give a characterization of the class $\mathcal{S}_1$. Recall that for $i\geq0$, $\mathcal{S}_i$ is the class of all loopless matroids $M$ with $\delta_0(M)=0$, $\delta_1(M)=i$.

\begin{proposition}\label{pr51}
 Let $M$ be a loopless matroid. If $M$ is connected, then the following conditions are equivalent:
\begin{enumerate}
 \item
 $M\in\mathcal{S}_1$;

\item
$M$ has a parallel irreducible decomposition: $M=P(N_1,\ldots,N_k)$, where $N_1$ is isomorphic to the cycle matroid of a subdivision of $K_{2,m}$ with $m\geq 3$, and $N_i$ is a non-loop circuit for $i=2,\ldots,k$.
\end{enumerate}
In general, $M\in\mathcal{S}_1$ if and only if $M$ can be decomposed as $M=M_1\oplus\cdots\oplus M_l$, where $M_1$ satisfies condition (ii) above, and $M_j$ is either a coloop or an iterated parallel connection of non-loop circuits for $j=2,\ldots,l$.
\end{proposition}

The proof of this result is based on the following easy consequence of Theorem \ref{th45}.

\begin{lemma}\label{lm52}
 Let $M$ be a parallel irreducible matroid. Then $M\in \mathcal{S}_1$ if and only if $M$ is isomorphic to the cycle matroid of a subdivision of $K_{2,m}$ with $m\geq 3$.
\end{lemma}

\begin{proof}
 First, assume that $M\in \mathcal{S}_1$. Let $G$ be a block with $M=M(G)$ and $\Pi=(\pi_1,\ldots,\pi_n)$ an ear decomposition of $G$. Evidently, $n\geq 2$. By Theorem \ref{th45} and Corollary \ref{co410}, $\Pi$ has a unique nest interval $I$ and $\ell(I)>1$. It follows that $I$ lies on the cycle $\pi_1$ and all other ears connect the two end vertices of $I$. Recall our convention that the length of $I$ does not exceed that of the path $\pi_1-I$. Thus $G$ consists of $n+1$ paths, each of which connects the two end vertices of $I$ and has length at least 2. Therefore, $G$ is a subdivision of $K_{2,n+1}$. Conversely, if $G$ is a subdivision of $K_{2,m}$ with $m\geq 3$, then $G$ has a nested ear decomposition with a unique nest interval $I$. Of course, $\ell(I)>1$. So $G$ is a series-parallel network and $\delta_1(M(G))=1$ by Theorem \ref{th45}. In other words, $M(G)\in \mathcal{S}_1$.
\end{proof}

\begin{proof}[Proof of Proposition \ref{pr51}]
The proof is a straightforward combination of Lemma \ref{lm25}(vii), Lemma \ref{lm26}, and Lemma \ref{lm52}.
\end{proof}

\begin{example}
 Consider the graph $G$ shown in Figure \ref{fig2}. We know from Example \ref{ex49} that $M(G)\in \mathcal{S}_1$. A parallel irreducible decomposition of $G$ is $G=P(G_1,G_2,G_3,G_4)$, where $G_1$, $G_2$, $G_3$, $G_4$ are subgraphs of $G$ induced by the edge sets $\{7,8,9,10,11,12\}$, $\{4,5,7\}$, $\{1,2,3,7\}$, $\{3,6\}$, respectively. We have $G_1\cong K_{2,3}$ and $G_i$ is a cycle for $i=2,3,4$.
\end{example}

Next, we characterize the class $\mathcal{S}_{1^+}=\mathcal{S}-\mathcal{S}_0$.

\begin{proposition}\label{pr54}
 Let $M\in\mathcal{S}$. Then the following conditions are equivalent:
\begin{enumerate}
 \item
 $M\in\mathcal{S}_{1^+}$;

\item
$M$ has a parallel minor isomorphic to $M(K_{2,m})$ for some $m\geq3$.
\end{enumerate}

\noindent If $M$ is connected, then each of the above conditions is equivalent to the following one:
\begin{enumerate}
\item[\rm(iii)]
$M$ has a series-parallel minor isomorphic to $M(K_{2,m})$ for some $m\geq3$.
\end{enumerate}
\end{proposition}

To prove this proposition, we will make use of the following lemma.

\begin{lemma}\label{lm55}
 Let $M\in\mathcal{S}_k$ with $k\geq0$. Suppose that $M$ is connected. Then there exists a sequence of matroids $M=M_k,M_{k-1},\ldots,M_0$ such that $M_i\in \mathcal{S}_i$ and $M_i$ is a series-parallel minor of $M_{i+1}$ for $i=0,\ldots,k-1$.
 \end{lemma}

\begin{proof}
 Let $G$ be a series-parallel network with $M=M(G)$. We argue by induction on the nullity $n(G)$ of $G$. If $n(G)=0$ or $n(G)=1$, then $M\in\mathcal{S}_0$ and we have nothing to prove. Assume $n(G)>1$. Let $\Pi$ be a nested ear decomposition of $G$. By Lemma \ref{lm46}, we may find a nest interval $I$ which is lined in $G$. Let $\Pi'=\Pi-\sigma(I)$ and let $G'$ be the subgraph of $G$ induced by the ears in $\Pi'$. Denote by $\tilde{G}$ and $\tilde{G}'$ the contractions of $G$ and $G'$ by all but one edge of $I$, respectively. Then $G$ is a subdivision of $\tilde{G}$, and $\tilde{G}$ is a parallel connection of $\tilde{G}'$ with cycles (see the proofs of Theorems \ref{th42}, \ref{th45}). This implies that $M(\tilde{G}')$ is a series-parallel minor of $M$. Since $n(\tilde{G}')<n(G)$ and $\delta_1(M)-1\leq \delta_1(M(\tilde{G}'))\leq \delta_1(M)$ (by Theorem \ref{th45}), the lemma now follows easily by the induction hypothesis.
\end{proof}

\begin{proof}[Proof of Proposition \ref{pr54}]
 Observe that $M$ satisfies condition (i) (respectively, (ii)) if and only if there is a connected component of $M$ satisfying the same condition. So we may assume that $M$ is connected.

(i)$\Rightarrow$(iii): By Lemma \ref{lm55}, $M$ has a series-parallel minor $M_1\in \mathcal{S}_1$. As $M$ is connected, $M_1$ is also connected by Lemma \ref{lm22}(ii). Now it follows easily from Proposition \ref{pr51} that $M_1$ has a series-parallel minor $N$ isomorphic to $M(K_{2,m})$ for some $m\geq3$. Evidently, $N$ is also a series-parallel minor of $M$.

(iii)$\Rightarrow$(ii): This is obvious.

(ii)$\Rightarrow$(i): Suppose $M$ has a parallel minor isomorphic to $M(K_{2,m})$ for some $m\geq3$. If $M\not\in\mathcal{S}_{1^+}$, then $M\in\mathcal{S}_0$. We will show that every loopless parallel minor of $M$ also belongs to $\mathcal{S}_0$. If $N$ is a parallel deletion of $M$, then by Lemma \ref{lm25}(ii), (v), the $h$-polynomials of $BC(M)$ and $BC(N)$ coincide. So $N\in\mathcal{S}_0$. Now consider the contraction of $M$ by an element $e$. By Lemma \ref{lm26}, $M$ is an iterated parallel connection of non-loop circuits (we may obviously exclude the case $M$ is a coloop). It then follows from Lemma \ref{lm22}(iii) that every connected component of $M/e$ is either a loop or an iterated parallel connection of non-loop circuits. Thus we can conclude that every loopless parallel minor of $M$ is in $\mathcal{S}_0$. This contradiction completes the proof.
\end{proof}

\begin{remark}
 Using Corollary \ref{co412} one can show that if $M\in \mathcal{S}$ and $N$ is a loopless parallel minor of $M$, then $N\in \mathcal{S}$ and $\delta_1(N)\leq \delta_1(M)$. This gives an alternative proof of the implication (ii)$\Rightarrow$(i) in Proposition \ref{pr54}.
\end{remark}

We are now in a position to give an excluded minor characterization of the class $\mathcal{S}_0$.

\begin{theorem}\label{th57}
 Let $M$ be a loopless matroid. Then $M\in \mathcal{S}_0$ if and only if $M$ has no minor isomorphic to $U_{2,4}$ or $M(K_4)$ and no parallel minor isomorphic to $M(K_{2,m})$ for all $m\geq3$.
\end{theorem}

\begin{proof}
 Since $U_{2,4}$, $M(K_4)$, $M(K_{2,m})$ are all connected, we may assume that $M$ is connected. The theorem now follows by combining Lemma \ref{lm24}(ii) and Proposition \ref{pr54}.
\end{proof}

The next result is a graph-theoretic version of the previous theorem. For a connected graph $G$, by abuse of notation, we also write $G\in \mathcal{S}_0$ (respectively, $G\in \mathcal{S}_{1}$, etc.) whenever $M(G)\in \mathcal{S}_0$ (respectively, $M(G)\in \mathcal{S}_{1}$, etc.). Thus $G\in \mathcal{S}_0$ if and only if $G$ is loopless and each block of $G$ is either an edge or an iterated parallel connection of non-loop cycles, by Lemmas \ref{lm22}(v) and \ref{lm26}. (Here, by a \emph{block of $G$} we mean a subgraph of $G$ that corresponds to a connected component of $M(G)$.) Recall that a \emph{vertex-induced subgraph} of $G$ is a graph obtained from $G$ by deleting a subset of the vertex set of $G$ together with all the edges incident to that vertex subset.

\begin{theorem}\label{th58}
 Let $G$ be a loopless connected graph with at least one edge. Denote by $\overline{G}$ the simplification of $G$. Then $G\in \mathcal{S}_0$ if  and only if $G$ has no subgraph that is a subdivision of $K_4$ and $\overline{G}$ has no vertex-induced subgraph that is a subdivision of $K_{2,3}$.
\end{theorem}

\begin{proof}
 We may assume that $G$ is a block. By Lemma \ref{lm24}(i), the condition that $G$ has no subgraph that is a subdivision of $K_4$ is equivalent to the fact that $G$ is a series-parallel network. Hence we need to prove that for a series-parallel network $G$, $G\in \mathcal{S}_0$ if and only if $\overline{G}$ contains no vertex-induced subgraph that is a subdivision of $K_{2,3}$. First, suppose $G\in \mathcal{S}_0$. We will show that every nonempty vertex-induced subgraph of $\overline{G}$ also belongs to $\mathcal{S}_0$. By Lemmas \ref{lm22}(i), (v) and \ref{lm26}, we may assume that $\overline{G}=P(D_1,\ldots,D_k)$, where the $D_i$ are simple cycles. Let $F$ be the set of baseedges of that parallel connection. For a vertex $v$ of $\overline{G}$, let $D_{j_1},\ldots,D_{j_l}$ be all the cycles containing $v$. Then using Lemma \ref{lm22}(iii) one may easily check that (see Figure \ref{fig4})
\[\begin{aligned}
M(\overline{G}-v)=&P(M(D_1),\ldots,M(D_{j_1-1}))\oplus M(D_{j_1}-v-F)\oplus P(M(D_{j_1+1}),\ldots,M(D_{j_2-1}))\\
&\oplus M(D_{j_2}-v-F)\oplus\cdots\oplus P(M(D_{j_l+1}),\ldots,M(D_k)).
\end{aligned}\]
Since $M(D_{j_i}-v-F)$ is either empty or a direct sum of coloops for $i=1,\ldots,l$, we deduce that $\overline{G}-v\in \mathcal{S}_0$. Consequently, every nonempty vertex-induced subgraph of $\overline{G}$ is in $\mathcal{S}_0$. Therefore, by Lemma \ref{lm52}, $\overline{G}$ has no vertex-induced subgraph that is a subdivision of $K_{2,3}$.

Now suppose $G\in \mathcal{S}_{1^+}$. We must show that $\overline{G}$ contains a vertex-induced subgraph that is a subdivision of $K_{2,3}$. The argument is by induction on the number $p(\overline{G})$ of nest intervals appearing in an ear decomposition of $\overline{G}$. If $p(\overline{G})=1$, then $\overline{G}\in \mathcal{S}_1$ and $\overline{G}$ is parallel irreducible by Corollary \ref{co410}. So from (the proof of) Lemma \ref{lm52}, $\overline{G}$ is a subdivision of $K_{2,m}$ for some $m\geq 3$. It follows that $\overline{G}$ has a vertex-induced subgraph that is a subdivision of $K_{2,3}$. Now consider the case $p(\overline{G})>1$. Let $\Pi$ be an ear decomposition of $\overline{G}$ and $I$ a nest interval which is lined in $\overline{G}$. The existence of $I$ is ensured by Lemma \ref{lm46}. Let $\overline{G}'$ be the subgraph of $\overline{G}$ induced by the ears in $\Pi':=\Pi-\sigma(I)$. Since $\overline{G}$ is simple, there is at most one path in $\sigma(I)^+=\sigma(I)\cup\{I\}$ which
has length 1. So by Remark \ref{rm43}, we may assume that all paths in $\sigma(I)$ have length at least 2. Then $\overline{G}'$ is a vertex-induced subgraph of $\overline{G}$. The desired conclusion now follows easily from the induction hypothesis.
\end{proof}

\begin{figure}[htb]
\begin{tikzpicture}[scale=1.5]
\draw (0,2)-- (0,1);
\draw (0,2)-- (1,2);
\draw (1,2)-- (0,1);
\draw (1,2)-- (0,3);
\draw (0,2)-- (0,3);
\draw (0,2)-- (0.51,2.49);
\draw (1,2)-- (2,1);
\draw (2,1)-- (1,0);
\draw (1,0)-- (0,1);
\draw (0,1)-- (0,0);
\draw (0,0)-- (1,0);
\draw (1,0)-- (2,0);
\draw (2,0)-- (2,1);
\draw (0.89,1.9) node[anchor=north west] {$v$};
\draw (0,0)-- (1,0);
\draw (0,0)-- (0,1);
\draw (0,1)-- (0,2);
\draw (0,2)-- (0,3);
\draw (3,0)-- (4,0);
\draw (4,0)-- (5,0);
\draw (3,1)-- (3,0);
\draw (4,0)-- (3,1);
\draw (5,1)-- (4,0);
\draw (5,1)-- (5,0);
\draw (5,0)-- (4,0);
\draw (3,3)-- (3,2);
\draw (3,2)-- (3,1);
\draw (3.57,2.55)-- (3,3);
\draw (3.57,2.55)-- (3,2);

\draw (0.8,-0.19) node[anchor=north west] {$\overline{G}$};
\draw (3.7,-0.23) node[anchor=north west] {$\overline{G}-v$};
\begin{scriptsize}
\draw (1.48,0.39) node[anchor=north west] {$D_1$};
\draw (0.84,1) node[anchor=north west] {$D_2$};
\draw (0.18,0.39) node[anchor=north west] {$D_3$};
\draw (0.1,1.8) node[anchor=north west] {$D_4$};
\draw (0.31,2.3) node[anchor=north west] {$D_5$};
\draw (0,2.66) node[anchor=north west] {$D_6$};

\draw (4.45,0.36) node[anchor=north west] {$D_1$};
\draw (3.17,0.38) node[anchor=north west] {$D_3$};
\draw (3.03,2.7) node[anchor=north west] {$D_6$};
\fill (0,0) circle (1.5pt);
\fill (1,0) circle (1.5pt);
\fill  (0,1) circle (1.5pt);
\fill (1,2) circle (1.5pt);
\fill (2,1) circle (1.5pt);
\fill  (2,0) circle (1.5pt);
\fill  (0,2) circle (1.5pt);
\fill (0,3) circle (1.5pt);
\fill  (0.51,2.49) circle (1.5pt);
\fill (4,0) circle (1.5pt);
\fill (3,0) circle (1.5pt);
\fill  (5,0) circle (1.5pt);
\fill (3,1) circle (1.5pt);
\fill (5,1) circle (1.5pt);
\fill (3,2) circle (1.5pt);
\fill (3,3) circle (1.5pt);
\fill (3.57,2.55) circle (1.5pt);
\end{scriptsize}
\end{tikzpicture}
 \caption{The class $\mathcal{S}_0$ is closed under vertex deletions.}\label{fig4}
\end{figure}
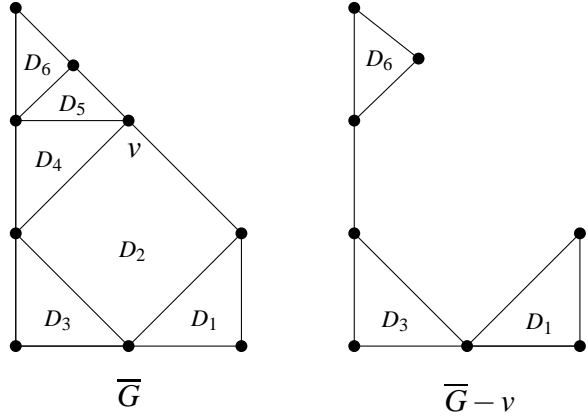

The above theorem shows a close relationship between the class $\mathcal{S}_0$ and outerplanar graphs. We say that a connected graph is \emph{outerplanar} if it can be embedded in the plane so that every vertex lies on the boundary of the infinite face. It was proved by Chartrand--Harary \cite[Theorem 1]{CH} that a graph is outerplanar if and only if it contains no subgraph that is a subdivision of $K_4$ or $K_{2,3}$.

\begin{corollary}\label{co59}
Every loopless outerplanar graph belongs to $\mathcal{S}_0$. Conversely, let $G$ be a graph in $\mathcal{S}_0$ with simplification $\overline{G}$. Then $G$ is outerplanar if and only if every block of $\overline{G}$ other than an edge is an iterated parallel connection of simple cycles with respect to pairwise distinct baseedges.
\end{corollary}

\begin{proof}
The first assertion follows immediately from Theorem \ref{th58} and the theorem of Chartrand--Harary mentioned above. Let us prove the second one. By Lemmas \ref{lm22}(i), (v) and \ref{lm26}, each block of $\overline{G}$ other than an edge is an iterated parallel connection $P(D_1,\ldots,D_k)$ of simple cycles. Let $e_i=D_{i+1}\cap(\bigcup_{j=1}^{i}D_j)$, $1\leq i\leq k-1$, be the baseedges of the parallel connection. If there are two non-distinct baseedges, say $e_1=e_2$, then $P(D_1,D_2,D_3)-e_1$ is a subgraph of ${G}$ which is a subdivision of $K_{2,3}$. So $G$ is not outerplanar. Now if all the baseedges $e_i$ are pairwise distinct, then one may easily embed the cycles $D_1,\ldots,D_k$ in the plane so that all vertices of $P(D_1,\ldots,D_k)$ lie on the boundary of the infinite face. It follows that $\overline{G}$, and hence $G$ as well, is outerplanar.
\end{proof}

\begin{remark}
 The subclass of $\mathcal{S}_0$ consisting of simple blocks, defined by a different characterization, was studied by McKee \cite{Mc}. In particular, he obtained results similar to Theorem \ref{th58} and Corollary \ref{co59}.
\end{remark}

The remaining part of this section is devoted to the study of $A$-graphs. Let $G$ be a connected graph. (Note that $G$ may contain loops.) We call $G$ an \emph{$A$-graph} if each block of $G$ other than an edge is an iterated parallel connection of non-loop cycles whose set of baseedges contains no cycle of $G$. Thus, in particular, loopless $A$-graphs are in $\mathcal{S}_0$. $A$-graphs were introduced by Fenton \cite{Fe} to characterize binary fundamental transversal matroids as well as a class of matroids which he called atomic (this prompted the name of $A$-graphs). It is proved in \cite[Theorems 3.5, 4.3]{Fe} that cycle matroids of $A$-graphs are exactly binary fundamental transversal matroids. This class of matroids is contained in the class of binary transversal matroids for which a characterization was found by Bondy \cite{Bo} and de Sousa--Welsh \cite{DW}: a matroid is binary transversal if and only if it is the cycle matroid of a graph which contains no subgraph that is a subdivision of $K_4$ or
$C_m^2$ for $m\geq 3$ (here $C_m^2$ is obtained from an $m$-cycle by replacing each edge with two parallel edges). Based on this characterization, it is conjectured in \cite{Fe} that $A$-graphs are precisely those graphs which contain no subgraph that is a subdivision of $K_4$, $K_{2,3}$ or $C_m^2$ for $m\geq 3$. As it stands, this conjecture is not true. For instance, the graph obtained from $K_{2,3}$ by adding a new edge between the two vertices of degree 3 is an $A$-graph. Nevertheless, in light of Theorem \ref{th58}, a slight modification of the conjecture does hold true:

\begin{corollary}\label{co511}
 Let $G$ be a connected graph with the simplification $\overline{G}$. Then the following conditions are equivalent:
\begin{enumerate}
 \item
$G$ is an $A$-graph;

\item
$G$ contains no subgraph that is a subdivision of  $K_4$ or $C^2_m$ for $m\geq 3$ and $\overline{G}$ contains no vertex-induced subgraph that is a subdivision of $K_{2,3}$;

\item
$G$ is planar and $G^*$ is an $A$-graph, where $G^*$ is a geometric dual of $G$;
\end{enumerate}
If $G$ is loopless and coloopless, then each of the above conditions is equivalent to the following one:
\begin{enumerate}
\item[\rm(iv)]
$G$ is planar and $G,G^*\in\mathcal{S}_0$.
\end{enumerate}
\end{corollary}

It is more convenient to prove first a matroid version of the above result. For brevity, cycle matroids of $A$-graphs will be called \emph{$A$-matroids}.

\begin{corollary}\label{co512}
 Let $M$ be a matroid. Denote by $M^*$ the dual of $M$. Then the following conditions are equivalent:
\begin{enumerate}
 \item
$M$ is an $A$-matroid;

\item
$M$ has no minor isomorphic to $U_{2,4}$ or $M(K_4)$, no series minor isomorphic to $M(C^2_m)$ for $m\geq 3$, and no parallel minor isomorphic to $M(K_{2,m})$ for $m\geq 3$;

\item
$M^*$ is an $A$-matroid;
\end{enumerate}
If $M$ is loopless and coloopless, then each of the above conditions is equivalent to the following one:
\begin{enumerate}
\item[\rm(iv)]
 $M,M^*\in \mathcal{S}_0$.
\end{enumerate}
\end{corollary}

\begin{proof}
 We may assume that $M$ is a loopless, coloopless connected matroid.

(i)$\Rightarrow$(ii): As mentioned before, $A$-matroids are binary transversal matroids; see \cite[Theorems 3.5, 4.3]{Fe}. So by \cite[Theorem 1]{Bo} and \cite[Theorem 1]{DW} (see also \cite[Theorem 13.4.8]{O}), $M$ does not contain $M(C^2_m)$ for $m\geq 3$ as series minors. On the other hand, since $M\in\mathcal{S}_0$, other excluded minors of $M$ as stated in (ii) come from Theorem \ref{th57}.

(ii)$\Rightarrow$(i): We first have $M\in\mathcal{S}_0$ by Theorem \ref{th57}. Let $G$ be a block with $M=M(G)$. By Lemmas \ref{lm22}(v) and \ref{lm26}, $G$ can be decomposed as an iterated parallel connection of non-loop cycles.
We must show that there exists such a decomposition whose set of baseedges contains no cycle of $G$. Let $\overline{F}(G)$ have the same meaning as in Theorem \ref{th42}. By Remark \ref{rm46}, we may find a decomposition of $G$ (as an iterated parallel connection of non-loop cycles) such that the set of baseedges is $\overline{F}(G)$. Note that $\overline{F}(G)$ contains no 2-cycle of $G$ by definition. On the other hand, $\overline{F}(G)$ also contains no $m$-cycle of $G$ for $m\geq 3$ since $M$ has no series minor isomorphic to $M(C^2_m)$. Therefore, we can conclude that $M$ is an $A$-matroid.

(i)$\Leftrightarrow$(iii): Note that $U_{2,4}$ or $M(K_4)$ is isomorphic to its dual, while $M(C^2_m)^*\cong M(K_{2,m})$ for $m\geq 3$. So from (i)$\Leftrightarrow$(ii) we immediately get (i)$\Leftrightarrow$(iii).

(i)$\Rightarrow$(iv): This follows easily from (i)$\Leftrightarrow$(iii).

(iv)$\Rightarrow$(ii): This follows from Theorem \ref{th57} and the fact that $M(C^2_m)^*\cong M(K_{2,m})$.
\end{proof}

Now we prove Corollary \ref{co511}.

\begin{proof}[Proof of Corollary \ref{co511}] The proof of (i)$\Leftrightarrow$(ii) is similar to that of (i)$\Leftrightarrow$(ii) in Corollary \ref{co512}, except that we now use Theorem \ref{th58} instead of Theorem \ref{th57}. It is clear that $A$-graphs are planar. So the equivalence of conditions (i), (iii) and (iv) follows from the equivalence of the corresponding conditions in Corollary \ref{co512}.
\end{proof}

It is proved in \cite[Theorem 1]{GMNN} that the Tutte polynomial characterizes the class of simple outerplanar graphs, in the sense that if two graphs $G_1,G_2$ have the same Tutte polynomial and $G_1$ is simple outerplanar, then $G_2$ is also outerplanar. A similar result holds for $A$-graphs.

\begin{proposition}
 Let $M$ be a loopless, coloopless $A$-matroid. If $N$ is a matroid with the same Tutte polynomial as $M$, then $N$ is also an $A$-matroid.
\end{proposition}

\begin{proof}
Let $M^*$ be the dual of $M$. Denote by $t(M;x,y)$ and $h(M;x)$ the Tutte polynomial of $M$ and the $h$-polynomial of $BC(M)$, respectively. Recall that $h(M;x)=t(M;x,0)$ and $h(M^*;x)=t(M;0,x)$; see \cite[p. 240]{B} and \cite[Proposition 6.2.4]{BrOx2}. So if $N$ has the same Tutte polynomial as $M$, then $h(M;x)=h(N;x)$ and $h(M^*;x)=h(N^*;x)$. Since the class $\mathcal{S}_0$ is characterized by the $h$-polynomial of the broken circuit complex (see Lemma \ref{lm26}), we conclude from Corollary \ref{co512} that $N$ is an $A$-matroid.
\end{proof}

\section{The nonnegativity of $\delta_2$}

The main aim of this section is to prove the following result.

\begin{theorem}\label{th61}
  Let $M$ be a matroid in $\mathcal{S}$. Then $\delta_2(M)\geq0$.
\end{theorem}

In the special case when $M\in\mathcal{S}_1$, a stronger statement holds true.

\begin{proposition}\label{pr62}
 Let $M\in\mathcal{S}_1$. Then $\delta_i(M)\geq0$ for all $i\geq0$.
\end{proposition}

Throughout this section, the notation $h_i(M)$ is used to denote the $(i+1)$-th entry of the $h$-vector of a broken circuit complex of $M$. We will need the following lemma.

\begin{lemma}\label{lm63}
For $M\in\mathcal{S}$, the following statements hold.
\begin{enumerate}
 \item
If $M\in\mathcal{S}_0$, then $h_0(M)\leq h_1(M)\leq\cdots\leq h_{\lfloor s/2\rfloor}(M)\geq h_{\lfloor s/2\rfloor+1}(M)\geq\cdots\geq h_s(M)$, where $s$ is the largest index such that $h_s(M)\ne0$.

\item
Assume $M$ is either the direct sum or the parallel connection of a matroid $N$ with a circuit. If $\delta_i(N)\geq0$ for all $i\geq0$, then $\delta_i(M)\geq0$ for all $i\geq0$.

\item
Assume $M$ is either the direct sum or the parallel connection of two matroids $N_1,N_2$. If $\delta_2(N_1),\delta_2(N_2)\geq0$, then $\delta_2(M)\geq0$.
\end{enumerate}
\end{lemma}

\begin{proof}
 Using Lemmas \ref{lm25}(ii) and \ref{lm26} we can reduce the proof of (i) to the case where $M$ is a circuit. The claim holds in this case by Lemma \ref{lm25}(v). The proof of (ii) also follows easily from Lemma \ref{lm25}(ii), (v). We now prove (iii). Suppose that $s,s_1,s_2$ are largest indices such that $h_s(M),h_{s_1}(N_1),h_{s_2}(N_2)\ne0$. We only consider the case $s_1,s_2\geq4$, the other cases are left to the reader. By Lemma \ref{lm25}(ii), $h_s(M)=h_{s_1}(N_1)h_{s_2}(N_2).$ Since $M\in\mathcal{S}$, $h_s(M)=h_{s_1}(N_1)=h_{s_2}(N_2)=1$. So from Lemma \ref{lm25}(ii) we get
\[\begin{aligned}
   h_2(M)&=h_2(N_1)+h_2(N_2)+h_1(N_1)h_1(N_2),\\
 h_{s-2}(M)&=h_{s_1-2}(N_1)+h_{s_2-2}(N_2)+h_{s_1-1}(N_1)h_{s_2-1}(N_2).
  \end{aligned}
\]
It follows that
\[\delta_2(M)=h_{s-2}(M)-h_2(M)=\delta_2(N_1)+\delta_2(N_2)+\epsilon\geq \epsilon,\]
where $\epsilon=h_{s_1-1}(N_1)h_{s_2-1}(N_2)-h_1(N_1)h_1(N_2)$. From $h_{s_i-1}(N_i)-h_1(N_i)=\delta_1(N_i)\geq0$ for $i=1,2$ we have $\epsilon\geq 0$. Therefore, $\delta_2(M)\geq0$.
\end{proof}

We now prove Proposition \ref{pr62}.

\begin{proof}[Proof of Proposition \ref{pr62}]
 From the characterization of the class $\mathcal{S}_1$ in Proposition \ref{pr51} and from Lemma \ref{lm63}(ii), we may assume that $M$ is parallel irreducible, i.e. $M$ is the cycle matroid of a graph $G$ which is a subdivision of $K_{2,m}$ with $m\geq3$. Denote by $G^+$ the graph obtained from $G$ by adding an edge $e$ between the two vertices of degree $m$ of $G$. Let $M^+=M(G^+)$ and $\tilde{M}=M^+/e$. Then $M^+$ is the parallel connection of $m$ circuits at the basepoint $e$ and $\tilde{M}$ is the direct sum of $m$ circuits. It follows that $M^+,\tilde{M}\in \mathcal{S}_0$. Let $s=r(M)-1$. Then by Lemma \ref{lm25}(iii), $s$ is the largest index with $h_s(M), h_s(M^+)\ne0$ and $s-m$ is the largest index with $h_{s-m}(\tilde{M})\ne 0$. We need to prove that $ \delta_i(M)\geq0$ for $i=1,\ldots,\lfloor s/2\rfloor$. Since $M=M^+-e$, from Lemma \ref{lm25}(iv) we get $h_i(M)=h_i(M^+)-h_{i-1}(\tilde{M})$ for $i=1,\ldots,s$. Hence for $i=1,\ldots,\lfloor s/2\rfloor$,
\[\begin{aligned}
   \delta_i(M)&=h_{s-i}(M)-h_i(M)=(h_{s-i}(M^+)-h_{s-i-1}(\tilde{M}))-(h_i(M^+)-h_{i-1}(\tilde{M}))\\
                      &=(h_{s-i}(M^+)-h_i(M^+))+(h_{i-1}(\tilde{M})-h_{s-i-1}(\tilde{M}))\\
                      &=h_{i-1}(\tilde{M})-h_{s-i-1}(\tilde{M}).
  \end{aligned}\]
The last equality follows since $M^+\in \mathcal{S}_0$; see Lemma \ref{lm26}. Now from $\tilde{M}\in \mathcal{S}_0$ we have $h_{i-1}(\tilde{M})=h_{s-m-i+1}(\tilde{M})$. Note that for $i=1,\ldots,\lfloor s/2\rfloor$, either $i-1$ or $s-m-i+1$ belongs to the interval $[\lfloor (s-m)/2\rfloor, s-i-1]$. Therefore, from Lemma \ref{lm63}(i) we conclude that $\delta_i(M)=h_{i-1}(\tilde{M})-h_{s-i-1}(\tilde{M})\geq0$.
\end{proof}

 For the proof of Theorem \ref{th61}, we will use the following calculation of $h_2(M)$ which was done in \cite[Theorem 5]{Bro}.

\begin{lemma}\label{lm64}
Let $M$ be a simple graphic matroid of rank $r$ on $n$ elements. Denote by $t(M)$ the number of $3$-circuits of $M$. Then $h_2(M)=\binom{n-r+1}{2}-t(M)$ .
\end{lemma}

\begin{proof}[Proof of Theorem \ref{th61}]
We use induction on the rank $r$ of $M$. If $r\leq 5$, then $\delta_2(M)=0$ by definition. Let $r\geq6$. By Lemma \ref{lm63}(iii), we may assume $M$ is parallel irreducible (thus, in particular, $M$ is simple). Furthermore, only the case $\delta_1(M)\geq2$ needs to be considered, by virtue of Proposition \ref{pr62}. Let $G$ be a block with $M=M(G)$. For an ear decomposition $\Pi$ of $G$, let $I$ be a nest interval which is lined in $G$. Such an $I$ exists by Lemma \ref{lm46}. Since $M$ is parallel irreducible, from Corollary \ref{lm43} we have $\ell(I)>1$. Let $\pi_i\in \sigma(I)$ and let $e\in\pi_i$. It is evident that every cycle of $G$ containing $e$ must have length at least 4. So $M/e$ is a simple connected matroid of rank $r-1$. Assume that $e$ is contained in $k$ $4$-cycles of $G$. Then $h_2(M)=h_2(M/e)+k$ by Lemma \ref{lm64}. On the other hand, from Lemma \ref{lm25}(iv) we get $h_{r-3}(M)=h_{r-3}(M-e)+h_{r-4}(M/e)$. Hence by the induction hypothesis,
\[\begin{aligned}
\delta_2(M)&=h_{r-3}(M)-h_2(M)= (h_{r-4}(M/e)-h_2(M/e))+(h_{r-3}(M-e)-k)\\
&=\delta_2(M/e)+(h_{r-3}(M-e)-k)\geq h_{r-3}(M-e)-k.
\end{aligned}\]
To complete the proof, we will show that $h_{r-3}(M-e)-k\geq0$. The case $k=0$ is obvious. Suppose now that $k>0$. Then $\pi_i$ must have length 2. Let $e'$ be the edge of $\pi_i$ other than $e$. Clearly, $M-e=\{e'\}\oplus M_1$, where $M_1$ is the cycle matroid of the graph $G_1$ obtained from $G$ by removing the ear $\pi_i$. By Lemma \ref{lm25}(iv), the $h$-vectors of broken circuit complexes of $M-e$ and $M_1$ coincide. This implies 
\[\delta_1(M_1)=\delta_1(M-e)=h_{r-3}(M-e)-h_1(M-e).\]
 (Note that $r-2$ is the largest index with $h_{r-2}(M-e)\ne0$ because $M-e$ has 2 connected components.) By Corollary \ref{co411}, $\delta_1(M_1)\geq \delta_1(M)-1\geq 1$. It follows that
 \[h_{r-3}(M-e)\geq h_1(M-e)+1=h_1(M-e)+h_0(M/e)=h_1(M).\]
 Now $h_1(M)$ is the nullity $n(G)$ of $G$ by Lemma \ref{lm25}(i). Since there are $k$ $4$-cycles of $G$ containing $e$, $\sigma(I)$ must have at least $k-1$ ears. We deduce that $n(G)\geq k$. Therefore, $h_{r-3}(M-e)\geq k$ and the proof of the theorem is complete.
\end{proof}

\noindent\textbf{Acknowledgements.} The author would like to thank Tim R\"{o}mer for helpful discussions and suggestions. He is grateful to the referees whose valuable comments substantially improved the paper and made it accessible to a broader audience. Thanks are also due to the Ministry of Education and Training of Vietnam and the German Academic Exchange Service (DAAD) for their financial support.

\end{document}